\newif\ifJ
\ifJ
\documentclass[review,onefignum,onetabnum,final]{siamart251216}
\usepackage{amsmath,amsfonts,amssymb}
\newsiamremark{remark}{Remark}
\newsiamremark{example}{Example}
\else

\documentclass[8pt]{article}
\usepackage[a4paper, total={5.5in, 8in}]{geometry}
\usepackage[numbers]{natbib}
\usepackage{amsmath,amsfonts,amsthm,amssymb}

\fi

\usepackage{mathrsfs}

\ifJ
\else
\usepackage[breaklinks,bookmarks=false]{hyperref}
\hypersetup{colorlinks, linkcolor=blue, citecolor=blue,
urlcolor=blue, plainpages=false, pdfwindowui=false,
pdfstartview={FitH}}
\fi

\usepackage{xcolor}
\usepackage{graphicx} 
\usepackage{mathtools}
\ifJ
\else
\fi
\usepackage{stmaryrd}
\newcommand{\jump}[1]{\left\llbracket #1 \right\rrbracket}
\usepackage{enumitem}
\usepackage{bigints}

\ifJ
\usepackage[disable]{todonotes} 
\else
\usepackage{todonotes}
\reversemarginpar
\fi

\DeclareMathOperator{\diam}{diam}

\newcommand{\abs}[1]{\left\lvert#1 \right\rvert}
\newcommand{\norm}[1]{\left\lVert#1\right\rVert}
\newcommand{\dualprod}[3]{\left\langle#1,#2\right\rangle_{#3}}
\newcommand{\LTwoInner}[3]{\left(#1,#2\right)_{#3}}
\newcommand{\dx}{\mathrm{d}x}

\newcommand{\dt}{\mathrm{d}t}
\newcommand{\dpH}{H_p}

\newcommand{\calF}{\mathcal{F}}
\newcommand{\calI}{\mathcal{I}}
\newcommand{\calR}{\mathcal{R}}

\newcommand{\calV}{\mathcal{V}}
\newcommand{\calS}{\mathcal{S}}
\newcommand{\R}{\mathbb{R}}
\newcommand{\N}{\mathbb{N}}
\newcommand{\V}{\mathbb{V}_{\T,\mathcal{J}}}
\newcommand{\Vone}{\mathbb{V}_1}
\newcommand{\Vtwo}{\mathbb{V}_2}
\newcommand{\Vi}{\mathbb{V}_i}

\newcommand{\RH}{\norm{R_1^X(\overline{u},\overline{m})}_{Y^*_0}}
\newcommand{\RK}{\norm{R_2^Y(\overline{u},\overline{m})}_{X^*}}

\newcommand{\timejump}[1]{\left( \! \left| #1 \right| \! \right)}

\newcommand{\uBar}{\overline{u}}
\newcommand{\mBar}{\overline{m}}

\newcommand{\osc}{\mathrm{osc}}
\newcommand{\etaRes}[1]{\eta_{R,#1}}

\newcommand{\etaRecon}[1]{\eta_{\mathrm{J},#1}}
\newcommand{\etaStab}[1]{\eta_{\mathcal{S},#1}}

\newcommand{\T}{\mathcal{T}}

\newcommand{\discreteFun}[3]{#1 _{#2,#3}}
\newcommand{\htau}[1]{\discreteFun{#1}{h}{\tau}}

\newcommand{\uht}{\htau{u}}
\newcommand{\mht}{\htau{m}}
\newcommand{\vht}{\htau{v}}
\newcommand{\wht}{\htau{w}}

\newcommand{\Uht}{\htau{U}}
\newcommand{\Mht}{\htau{M}}

\newcommand{\taun}{\tau_n}

\newcommand{\massLumped}[2]{\LTwoInner{#1}{#2}{\Omega,h}}
\newcommand{\TF}{\T_{F}}
\newcommand{\calFKi}{\calF_K^{I}}
\newcommand{\tTK}{\widetilde{\T}_K}
\newcommand{\calE}{\mathcal{E}}

\numberwithin{equation}{section}

\ifJ
\else
\newtheorem{theorem}{Theorem}[section]{\bfseries}{\it}
\newtheorem{proposition}[theorem]{Proposition}{\bfseries}{\it}
\newtheorem{lemma}[theorem]{Lemma}{\bfseries}{\it}
{\bfseries}{\it}
{\bfseries}{\it}
\newtheorem{example}{Example}[section]{\bfseries}{\rmfamily}
{\bfseries}{\it}
\newtheorem{remark}{Remark}[section]{\bfseries}{\rmfamily}
\newcommand{\proofbox}{\qed}
\fi

\ifJ

\headers{A posteriori error bounds for MFG}{I.~Smears \& H.~Wells}

\title{A posteriori error bounds for finite element approximations of time-dependent mean field games\thanks{Submitted to the editors xx-yy-2026}.\funding{I.~Smears and H.~Wells were supported by the Engineering and Physical Sciences Research Council [grant number EP/Y008758/1]}}

\author{Iain Smears\thanks{Department of Mathematics, University College London, Gower
	Street, WC1E 6BT London, United Kingdom. (\email{i.smears@ucl.ac.uk}, \email{h.wells@ucl.ac.uk}).}
\and Harry Wells\footnotemark[2]
}

\else

\title{A posteriori error bounds for finite element approximations of time-dependent mean field games}
\author{Iain Smears, Harry Wells}

\fi

\begin{document}
\maketitle

\begin{abstract}
We present \emph{a posteriori} error bounds for a general class of stabilized finite element approximations of time-dependent mean field games.
We first show the equivalence between the norm of the error and the dual norm of the residual in the coupled Hamilton--Jacobi--Bellman and Kolmogorov--Fokker--Planck equations.
We then derive a reliable and efficient \emph{a posteriori} error estimator that is based on residual estimators, along with the temporal jump estimator, and an estimator for the stabilization terms in the numerical discretization.
Finally, for stabilizations based on mass-lumping in time and affine-preserving spatial stabilizations, we show that the stabilization estimators can be bounded in terms of the residual and temporal jump estimators, thus yielding an improved reliable, locally computable, and locally efficient estimator.
\end{abstract}

\ifJ
\fi

\section{Introduction}\label{sec-introduction}
Mean field games (MFG) model the Nash equilibria of differential games for a large population of players. MFG were introduced by Lasry and Lions \cite{lasry2006jeuxI,lasry2006jeuxII,lasry2007mean} and independently by Huang, Caines, and Malham\'{e} \cite{huang2006large}.
In this paper, we consider second-order time-dependent MFG systems of the form
\begin{equation}\label{eq:MFG-system}
\begin{aligned}
    -\partial_t u -\nu \Delta u + H(t,x,\nabla u) &= F[m] \ && \mathrm{in }~ (0,T) \times \Omega, \\
    \partial_t m -\nu \Delta m - \mathrm{div}\left( \dpH(t,x,\nabla u) ~m \right) &= G \ && \mathrm{in}~ (0,T) \times \Omega, \\
    m(0) = m_0,\quad u(T) &= S[m(T)] \ && \mathrm{in}~ \Omega,   \\
    m = 0, \quad u&=0 \ &&\mathrm{on}~ (0,T) \times \partial \Omega ,
\end{aligned}
\end{equation}
where $\Omega\subset \R^d$ is a bounded domain that represents the state space of the game, where~$u$ denotes the value function, and where~$m$ denotes the density of players across the state space~$\Omega$. 
The main assumptions in this work regarding the problem data are that $\nu>0$ is a constant, the Hamiltonian $(t,x,p)\mapsto H(t,x,p)$ is~$C^{1,1}$ with respect to~$p$, the coupling terms~$F$ and~$S$ are Lipschitz continuous and monotone operators on suitably chosen function spaces, the source term~$G$ is nonnegative in the sense of distributions, and the initial density~$m_0$ is nonnegative. More detailed assumptions on the problem data are stated in Section~\ref{sec:preliminaries} below.

The numerical approximation of MFG presents several significant challenges, such as the nonlinear coupling of the forward and backward parabolic equations, the lack of coercivity of the spatial differential operators, and the need to preserve nonnegativity of the density at the discrete level.
Finite element methods (FEM) for time-dependent MFG systems, such as~\eqref{eq:MFG-system}, and their steady-state counterparts, were analysed recently by Osborne and the first author in~\cite{OsborneSmears2024i,OsborneSmears2025i}, where the convergence of the methods was proved for MFG systems with nondifferentiable Hamiltonians. 
The FEM in these works are based on a continuous piecewise affine discretization with stabilization of the advective terms in space, and an implicit Euler discretization of the temporal derivatives with lumped masses, which are chosen to ensure a discrete maximum principle and the nonnegativity of the approximations of the density~$m$.
The asymptotic near quasi-optimality and optimal convergence rates of the FEM were then proved in~\cite{OsborneSmears2025ii} for steady-state problems where the Hamiltonian $H$ is $C^{1,1}$ with respect to the gradient variable of the value function, and the couplings are monotone.
Convergence rates of FEM for steady-state MFG where the Hamiltonian $H$ is merely Lipschitz have also recently been shown in~\cite{OsborneSmears2026}, based on the regularization analysis from~\cite{OsborneSmears2025iii}.
Berry, Ley \& Silva~\cite{BerryLeySilva2025} have analysed convergence rates of FEM when the solutions are stable with respect to linearizations.
This approach has recently been extended to spatial semidiscretizations for time-dependent problems by Berry in~\cite{berry2026}.

Whereas the works above concern the \emph{a priori} analysis of FEM for MFG, we are interested here in \emph{a posteriori} analysis, where the goal is to bound the error in terms of a computable error estimator that depends on the numerical solution and the computational meshes, without requiring \emph{a priori} knowledge of the true solution or assumptions about its regularity.
\emph{A posteriori} error estimators are a central ingredient for adaptive mesh refinement strategies, which can be significantly more computationally efficient than uniform mesh refinements.
We refer the reader to the textbook~\cite{verfurth2013posteriori} for an introduction to \emph{a posteriori} error analysis.
We emphasize also that the \emph{a posteriori} error analysis of parabolic problems features a number of challenges not present in the elliptic setting, especially regarding the efficiency of the estimators and the lack of temporal conformity of standard temporal discretizations, see~\cite{ErikssonJohnson1991,ErnSmearsVohralik2017b,ErnSmearsVohralik2019,GeorgoulisMakridakis2023,LakkisMakridakis2006,MakridakisNochetto2006,Picasso1998,Smears2026,Verfurth2003} for a variety of approaches in the case of linear parabolic problems, and also to~\cite{smears2025IntroductionAPosteriori} for an introductory overview.
The only work so far on \emph{a posteriori} error analysis for numerical discretizations of MFG is~\cite{OsborneSmearsWells2025}, where reliable and efficient estimators were established for a broad class of stabilized FEM in the case of steady-state MFG systems with $C^{1,1}$ Hamiltonians.
The estimators consist of residual-based estimators plus an additional stabilization estimator that results from the stabilization terms in the method.
It was also shown in~\cite{OsborneSmearsWells2025} that, if the stabilization terms have some additional structure, including being locally affine-preserving, then the stabilization estimator can be bounded in terms of the jump estimators, so it does not need to be computed in practice.

The starting point for the \emph{a posteriori} error analysis of~\eqref{eq:MFG-system} is the well-posedness of the continuous problem.
It is known already that, under suitable assumptions such as the monotonicity of the coupling operators~$F$ and~$S$ and nonnegativity of $m_0$ and $G$, the MFG system is well-posed with a unique weak solution pair $(u,m)\in X\times Y$ where the Bochner--Sobolev spaces are defined by $X\coloneqq L^2(0,T;H^1_0(\Omega))$ and $Y\coloneqq X\cap H^1(0,T;H^{-1}(\Omega))$; we refer to~\cite[Theorems 3.1 \& 3.2]{OsborneSmears2025i} and~\cite[Theorem 5.2.1]{osborne2024Thesis}, as well as Section~\ref{sec:preliminaries} below, for details.
In the first main result of this work, see Theorem~\ref{theorem:equivalence-result} below, we prove the stronger quantitative property, namely the equivalence between error and dual norm of residual, which takes the form 
\begin{equation}
\norm{u-\uBar}_X + \norm{m-\mBar}_Y \eqsim \calR(\uBar,\mBar), \label{eq:intro-equiv}
\end{equation}
for all $(\uBar,\mBar)\in X\times Y$ such that $\mBar \geq 0$ a.e.\ in $(0,T)\times \Omega$, where the residual functional $\calR(\cdot,\cdot)$, defined in~\eqref{eq:calR-def} below, comprises the dual norms of the residuals of the weak forms of the HJB and KFP equations, see~\eqref{eq:R_i-defs}, alongside the initial condition error of the KFP equation.
The notation $\eqsim$ in~\eqref{eq:intro-equiv} means that the left- and right-hand sides are bounded from above and below by each other, up to hidden multiplicative constants that depend solely on the problem data.
We note that one improvement in the analysis here for the time-dependent MFG system over that of the steady-state problem in~\cite{OsborneSmearsWells2025} is that the equivalence result~\eqref{eq:intro-equiv} is not restricted to functions $(\uBar,\mBar)$ in some neighbourhood of the solution, so the resulting \emph{a posteriori} error bounds for time-dependent problems are applicable also on coarse meshes.

We then derive \emph{a posteriori} error bounds for a broad family of discretizations of~\eqref{eq:MFG-system}, consisting of conforming piecewise affine FEM in space and implicit Euler discretizations in time, along with general abstract stabilizations to ensure nonnegativity of the density. This class includes the method of~\cite{OsborneSmears2025i} as a particular example.
Methods in this family lead to numerical approximations $(\uht,\mht)\in \Vone\times \Vtwo$, where the discrete spaces $\Vi$ are discrete approximations spaces of functions that are piecewise constant in time with respect to the time-step partition of $[0,T]$, and continuous piecewise affine in space with respect to a mesh $\mathcal{T}$ over $\Omega$, see Section~\ref{sec:FEM-scheme} below for further details.
Since the approximation spaces $\Vi$ are contained in the space~$X$ but not in~$Y$, the equivalence~\eqref{eq:intro-equiv} bound cannot be immediately applied to the numerical solution $(\uht,\mht)$.
Following~\cite{ErnSmearsVohralik2017b}, we address this challenge by defining a suitable extension of the norm on $Y$ to the sum space $\Vi+Y$, see~\eqref{eq:error-norm-defs} below, and we apply the equivalence result~\eqref{eq:MFG-system} to the pair $(\uht,\Mht)$ where $\Mht\in Y$ is a suitably defined continuous piecewise affine in time reconstruction based on $\mht$. 
These extended norms and reconstructions bridge the temporal conformity gap, enabling us to use the continuous equivalence~\eqref{eq:intro-equiv} to bound the errors of the stabilized FEM solution $(\uht,\mht) \in \Vone \times \Vtwo$.
This leads to our main results on the \emph{a posteriori} error bounds in Theorems~\ref{theorem:reliability}, \ref{theorem:local-efficiency}, and~\ref{theorem:global-efficiency} below, where we show the reliability and efficiency of the estimators via bounds of the form
\begin{equation}\label{eq:intro:a-post-bounds}
    \begin{aligned}
    \norm{u-\uht}_{\Vone + Y} + \norm{m-\mht}_{\Vtwo + Y} \lesssim \eta(\uht,\mht) + \text{ oscillation}, \\
    \eta(\uht,\mht) \lesssim \norm{u-\uht}_{\Vone + Y} + \norm{m-\mht}_{\Vtwo + Y} + \text{ oscillation},
\end{aligned}
\end{equation}
where $\eta(\uht,\mht)$ is a computable estimator that includes the standard residual estimators, the temporal jump estimators that measure the lack of $Y$-conformity of the numerical approximations, stabilization estimators, and estimators for the error in the initial and final time conditions.
As above, the notation $\lesssim$ means that the inequality holds with a hidden multiplicative constant that is independent of the discretization parameters.
The oscillation terms in~\eqref{eq:intro:a-post-bounds} represent the data approximation error terms that typically arise in residual-type \emph{a posteriori} error analysis.
The estimator and data oscillation terms are defined in Section~\ref{sec:estimator-defs}.
Therefore, the bounds in~\eqref{eq:intro:a-post-bounds} show that the resulting estimators are reliable, as well as globally efficient.

For the general class of abstract stabilizations for which we show~\eqref{eq:intro:a-post-bounds}, the stabilization estimator is not locally computable. 
Therefore, in order to obtain a locally computable and locally efficient estimator, we refine the analysis for a narrower class of stabilizations, namely those that consist of mass lumping of the time derivative terms with the patchwise affine-preserving stabilization of the spatial terms.
First, we show that the part of the stabilization estimator that results from mass lumping is bounded by the standard temporal jump estimator, see Proposition~\ref{prop:mass-lump-estimate} below, under the hypothesis that $h^2\lesssim \tau$ where $h$ is the spatial mesh-size and $\tau$ is the time-step size. 
Note that the condition $h^2\lesssim \tau$ corresponds to the practically relevant case in computations, as the time-steps are usually not excessively small compared to the mesh-size.
Furthermore, we show that, for patchwise affine-preserving stabilizations, the spatial stabilization terms are bounded by the spatial jump estimator (which is one of the components of the standard residual estimator), see Proposition~\ref{prop:stab<=jump} below. 
Consequently, we show in Theorem~\ref{theorem:stab-free-error-bound} below that the stabilization estimators can be removed, and the whole error is bounded (up to data oscillation terms) by the residual, temporal jump, initial, and final time estimators. 
Therefore, we obtain reliable and locally efficient \emph{a posteriori} error estimators for this class of methods.

The remainder of this paper is organized as follows. Section~\ref{sec:preliminaries} introduces the notation and fundamental properties of the MFG system~\eqref{eq:MFG-system} required for our analysis. 
Section~\ref{sec:equivalances} establishes the continuous-level equivalence between the approximation errors and the dual norms of the residuals.
The general class of stabilized FEM schemes considered in this work is defined in Section~\ref{sec:FEM-scheme}. 
In Section~\ref{sec:a-post-bounds}, we construct the \emph{a posteriori} error estimators and prove their reliability and efficiency.
Finally, Section~\ref{sec:local-estimators} derives stabilization-free error bounds for a class of stabilization schemes with additional structure.

\section{Preliminaries} \label{sec:preliminaries}

\paragraph{Basic notation}
We denote $\mathbb{N}\coloneqq \{1,2,3,\cdots\}$.
For a Lebesgue measurable set $\omega\subset \R^d$, $d\in\N$, we let $|\omega|_d$ denote its Lebesgue measure and $\diam~ \omega$ its diameter.
Let $L^2(\omega)$ and $L^2(\omega;\R^d)$ denote respectively the usual Lebesgue spaces of scalar, respectively $d$-dimensional vector fields, on $\omega$.
The inner products on $L^2(\omega)$ and $L^2(\omega;\R^d)$ are denoted commonly by~$(\cdot,\cdot)_{\omega}$, with induced norm denoted by $\norm{\cdot}_\omega$; there will be no risk of confusion as the two cases will be distinguished by the arguments.
Let $L^{\infty}(\omega;\mathbb{R}^{d\times d})$ denote the space of essentially bounded $d\times d$ matrix-valued functions on $\omega$, equipped with the essential supremum norm $\|\cdot\|_{L^{\infty}(\omega;\mathbb{R}^{d\times d})}$ that is induced by the Frobenius norm on $\mathbb{R}^{d\times d}$.
Let $\Omega$ be a bounded, connected, open subset of $\mathbb{R}^d$, $d\in\N$, with Lipschitz boundary $\partial \Omega$. 
For a given final time $T>0$, let $I \coloneqq (0,T)$ denote the bounded open time interval,  and let $Q_T = I \times \Omega$ denote the time-space domain. 
Let $H^1(\Omega)$ and $H^1_0(\Omega)$ denote the usual Sobolev spaces, cf.~\cite{AdamsFournier03}.
Note that the Poincar\'e inequality for the bounded domain $\Omega$ implies that $v\mapsto \norm{\nabla v}_\Omega$ defines a norm on $H^1_0(\Omega)$.
Let $H^{-1}(\Omega)$ denote the dual space of $H_0^1(\Omega)$. 
The canonical norm on $H^{-1}(\Omega)$ is defined by
\begin{equation}
\norm{\Phi}_{H^{-1}(\Omega)}\coloneqq \sup_{\substack{v\in H^1_0(\Omega)\\\norm{\nabla v}_\Omega =1}} \langle \Phi, v\rangle_{} \quad \forall\Phi \in H^{-1}(\Omega),
\end{equation}
where $\langle\cdot,\cdot\rangle_{}$ denotes the duality pairing between $H^{-1}(\Omega)$ and $H^1_0(\Omega)$.

\paragraph{Notation for inequalities}
In the following, the notation $x \lesssim y$ will indicate the existence of a constant $C >0$ such that $x \leq C y$. If there exist two constants such that $x \lesssim y$ and $y \lesssim x$ we will denote this relation by $x \eqsim y$. We shall, when necessary, list the dependency of the hidden constant but not explicitly define it. Hidden constants will typically depend on the PDE data, the domain, and the quality of the domain discretization. Hidden constants will not depend on any discretization parameters.

\paragraph{Bochner spaces}
In the following, we will consider several Bochner--Sobolev spaces~\cite[Chapter IV]{Wloka_1987}. 
Let $X\coloneqq L^2(I;H^1_0(\Omega))$, which is a Hilbert space when equipped with the inner-product $(w,v)\mapsto\int_0^T(\nabla w, \nabla v)_\Omega\,\mathrm{d} t$ and norm
\begin{equation}
\norm{w}_X^2 \coloneqq \int_0^T \norm{\nabla w}^2_\Omega\, \dt  \quad\forall w\in X. \label{eq:X-norm-def}
\end{equation}
Let $X^*$ denote the dual space of $X$ and let $\dualprod{\cdot}{\cdot}{X^* \times X}$ denote the natural duality pairing. 
We define the dual norm on $X^*$  by
\begin{equation}
    \norm{f}_{X^*} \coloneqq \sup_{w \in X \backslash \{0\}} \frac{\dualprod{f}{w}{X^* \times X}}{\norm{w}_X} \quad \forall f \in X^*.
\end{equation}
It is known that the space $X^*$ can be identified with the space $L^2(I; H^{-1}(\Omega))$~\cite[Section 26]{Wloka_1987}.

We also define the space $Y\coloneqq L^2(I;H^1_0(\Omega))~\cap~H^1(I;H^{-1}(\Omega))$, which consists of all functions $v\in X$ that are weakly differentiable in time with time derivative $\partial_t v \in L^2(I;H^{-1}(\Omega))$. 
It is known that the space $Y$ can be equipped with the norm $\norm{\cdot}_Y$ defined by
\begin{equation}\label{eq:Y-norm-def}
\norm{v}_Y^2 \coloneqq \int_0^T \norm{\partial_t v}_{H^{-1}(\Omega)}^2 + \norm{\nabla v}^2_\Omega\ \dt + \norm{v(T)}_\Omega^2 \quad \forall v \in Y, 
\end{equation}
where $v(T)$ denotes the trace value of $v$ at time $T$ in $L^2(\Omega)$.
In particular, the norm $\norm{\cdot}_Y$ is well-defined since the space $Y$ is continuously embedded into $C([0,T]; L^2(\Omega))$.
Furthermore, we have the embedding bound
\begin{equation}
 \max_{t \in [0,T]} \norm{v(t)}_\Omega \leq \norm{v}_Y \quad \forall v \in Y. \label{eq:Y-embedding}
\end{equation}
Additionally, we define $Y_0$ as the space of functions in $Y$ that vanish at initial time, i.e.\ $Y_0 \coloneqq \{ v \in Y\ :\ v(0)=0 \}$.

\paragraph{Problem data}
Let the diffusion coefficient $\nu >0$ be constant.
The initial distribution $m_0$ is assumed to be in $L^\infty(\Omega)$ with $m_0 \geq 0$ a.e.\ in $\Omega$.

The Hamiltonian of the underlying optimal control problem is defined by
\begin{equation}
    H(t,x,p) \coloneqq \sup_{\alpha \in \mathcal{A}} \left[ b(t,x,\alpha) \cdot p - f(t,x,\alpha) \right] \quad \forall (t,x,p) \in \overline{Q_T} \times \R^d,  
     \label{eq:H-def}
\end{equation}
where it is assumed that the control set $\mathcal{A}$ is a compact metric space, and the control-dependent drift $b:[0,T] \times \Omega \times \mathcal{A} \rightarrow \R^d$ and control-dependent running cost $f:Q_T \times \mathcal{A} \rightarrow \R$ are uniformly continuous on $\overline{Q_T} \times \R^d$. 
It follows from these assumptions that the Hamiltonian $H$, defined in~\eqref{eq:H-def} above, is convex and Lipschitz continuous in its third argument, i.e.\
\begin{equation}\label{eq:H-cts-condition}
    \abs{H(t,x,p) - H(t,x,q)} \leq L_H \abs{p-q} \quad \forall (t,x,p,q) \in \overline{Q_T} \times \R^d \times \R^d,
\end{equation}
where $L_H\coloneqq \norm{b}_{C\left(\overline{Q_T} \times \mathcal{A}; \R^d \right)}$. 
We further assume that $H$ is $C^{1,1}$ on $\R^d$ with respect to its third argument, i.e.\ $\dpH$ exists for all arguments, and there exists a constant~$L_{H_p}$ such that
\begin{equation} \label{eq:dpH-cts-condition}
    \abs{\dpH(t,x,p) - \dpH(t,x,q)} \leq L_{H_p} \abs{p-q} \quad \forall (t,x,p,q) \in \overline{Q_T} \times \R^d \times \R^d.
\end{equation}
The Lipschitz continuity of $H$ in~\eqref{eq:H-cts-condition} implies that
\begin{equation}\label{eq:dpH-bounded-condition}
    \abs{\dpH(t,x,p)} \leq L_H \quad \forall (t,x,p) \in \overline{Q_T} \times \R^d.
\end{equation}
Let $D_H(\cdot,\cdot): \overline{Q_T} \times \R^d \times \R^d \to \R$ denote the Bregman divergence of the Hamiltonian with respect to its third argument, defined as
\begin{equation} \label{eq:Bregman-div-def}
    D_H(t,x,p,q) \coloneqq H(t,x,p) - H(t,x,q) - \dpH(t,x,q) \cdot (p-q) \quad \forall (t,x,p,q)\in \overline{Q_T} \times \R^d \times \R^d.
\end{equation}
As a consequence of the convexity of the Hamiltonian and the Lipschitz continuity condition~\eqref{eq:dpH-cts-condition}, we have the following well-known inequality, which can be found for instance in~\cite[Theorem~2.1.5, p.~67]{Nesterov2018},
\begin{equation}\label{eq:Bregman-inequality}
    \abs{\dpH(t,x,p)-\dpH(t,x,q)}^2 \leq 2 L_{H_p} D_H(t,x,p,q) \quad  \forall (t,x,p,q) \in \overline{Q_T}\times \R^d\times \R^d.
\end{equation}
To abbreviate the notation, for a function $v\in X$, let $H[\nabla v]\in L^2(0,T;L^2(\Omega))$ denote the composition of $H$ with $\nabla v$, i.e.\ $H[\nabla v](t,x)\coloneqq H(t,x,\nabla v(t,x))$ a.e.\ $t\in (0,T)$, a.e.\ $x\in \Omega$.
In a similar manner, for a pair $(w,v)\in X\times X$, let $\dpH[\nabla v]$ and $D_H[\nabla w,\nabla v]\in L^2(0,T;L^2(\Omega)) $ be defined by $\dpH[\nabla v](t,x)\coloneqq  \dpH(t,x,\nabla v(t,x))$ and $D_H[\nabla w,\nabla v](t,x)\coloneqq D_H(t,x,\nabla w(t,x),\nabla v(t,x))$ for a.e.\ $t\in (0,T)$, a.e.\ $x\in \Omega$.

We now specify the assumptions on the coupling operators $F$ and $S$. We suppose that there exists a Hilbert space $Z$ such that the nonlinear coupling operator $F\colon Z \to X^*$ is Lipschitz continuous, i.e.\ there exists a constant $L_F >0$ such that 
\begin{align}
     \norm{F[v_1] - F[v_2]}_{X^*} &\leq L_F \norm{v_1 -v_2}_{Z} \quad \forall v_1,v_2 \in Z. \label{eq:F-cts-condition}
\end{align}  
Similarly, we assume that the final time coupling operator $S:L^2(\Omega) \to L^2(\Omega)$ is Lipschitz continuous, i.e.\ there exists a constant $L_S>0$ such that
\begin{equation}
     \norm{S[v_1] - S[v_2]}_\Omega \leq L_S \norm{v_1 - v_2}_\Omega \label{eq:S-cts-condition} \quad \forall v_1,v_2 \in L^2(\Omega).
\end{equation}  
Furthermore, we assume that the spaces $X$ and $Y$ are continuously embedded in $Z$, with the embedding of $Y \hookrightarrow Z$ moreover compact. 
We assume that the coupling operators $F$ and $S$ satisfy a monotonicity condition:
\begin{align}\label{eq:F&S-monotonicity-condition}
    0 \leq \int^T_0 \dualprod{F[v_1]-F[v_2]}{v_1-v_2}{} \dt + \LTwoInner{S[v_1(T)]-S[v_2(T)]}{v_1(T)-v_2(T)}{\Omega},  
\end{align}
for all $v_1, v_2 \in Y$. Note the monotonicity condition~\eqref{eq:F&S-monotonicity-condition} is only required on the space~$Y$, even for coupling terms defined over a larger Hilbert space~$Z$.

Let $G \in L^2(I; H^{-1}(\Omega))$ be of the form $G = g_0 - \nabla \cdot g_1$, where $g_0 \in L^{r}(I;L^{s}(\Omega))$ and $g_1 \in L^{2r}(I;L^{2s}(\Omega;\R^d))$ for some indices $r,s \in (1,\infty]$ satisfying 
\begin{equation}\label{eq:G-indices-conditon}
     \frac{1}{r} + \frac{d}{2s} < 1.
 \end{equation}
 Furthermore, we assume that the source term $G$ is nonnegative in the sense of distributions, i.e.\ $\int_0^T \dualprod{G}{w}{} \dt \geq 0$ for all $w \in X$ which satisfy $w \geq 0$ a.e. in $Q_T$. 

\paragraph{Weak formulation}
A weak formulation of the MFG system~\eqref{eq:MFG-system} is as follows: find $(u,m) \in X \times Y$ such that $m(0) = m_0$ and
\begin{subequations} \label{eq:weak-formulation-XY}
\begin{flalign}
     & \int_0^T \left[ \dualprod{\partial_t v}{u}{} + \LTwoInner{\nu \nabla u}{\nabla v}{\Omega} + \LTwoInner{H[\nabla u]}{v}{\Omega} \right]\dt \label{eq:weak-formulation-HJB}
     \\
      & \hspace{4.4cm}=\int_0^T \dualprod{F[m]}{v}{}\dt + \LTwoInner{S[m(T)]}{v(T)}{\Omega}, \notag
    \\ &\int_0^T \left[ \dualprod{\partial_t m}{\phi}{} + \LTwoInner{\nu \nabla m}{\nabla \phi}{\Omega} + \LTwoInner{m \dpH[\nabla u]}{\nabla \phi}{\Omega} \right]\dt
    = \int_0^T \dualprod{G}{\phi}{}\dt, \label{eq:weak-formulation-KFP}
\end{flalign}
\end{subequations}
for all $(v,\phi) \in Y_0 \times X$. 
Notice that the weak formulation~\eqref{eq:weak-formulation-XY} is obtained by integration-by-parts in time of the temporal derivative term in the HJB equation. 
Under the hypotheses on the problem data given above, in \cite[Theorems 3.1 \& 3.2]{OsborneSmears2025i}, the existence and uniqueness of a weak solution of~\eqref{eq:weak-formulation-XY} was shown for the case when $Z = L^2(I; L^2(\Omega))$, without requiring differentiability of $H$. Following the proof of \cite[Theorem 5.2.1]{osborne2024Thesis}, it is clear that existence of a solution also holds when $F$ is defined on the more general space $Z$ that satisfies the assumptions stated above.
Note that the uniqueness of the solution follows from the  nonnegativity of $m_0$ and $G$, and the monotonicity condition~\eqref{eq:F&S-monotonicity-condition} on the coupling terms $F$ and $S$. 
Furthermore, the density $m$ is nonnegative a.e. in $Q_T$, and is additionally essentially bounded, i.e.\
\begin{align}
    \norm{m}_{L^\infty(Q_T)} &\leq M_\infty, \label{eq:M_infty}
\end{align}
where $M_\infty$ is a constant that depends only~$d$, the indices~$r$ and~$s$ that satisfy~\eqref{eq:G-indices-conditon}, on~$\nu$, $L_H$, the measure of $\Omega$, the $L^r(I;L^q(\Omega))$-norm of~$g_0$, and the $L^{2r}(I;L^{2q}(\Omega;\R^d))$-norm of~$g_1$.
A proof of~\eqref{eq:M_infty} can be found in~\cite[Theorem 7.1, Chapter III]{ladyzhenskaia1968Parabolic}.
Finally, we note that the hypotheses that~$F$ and~$S$ respectively take values in~$X^*$ and in $L^2(\Omega)$ imply that the value function~$u$ that solves~\eqref{eq:weak-formulation-HJB} is also in the space $Y$, i.e.\ the temporal derivative $\partial_t u$ exists in $L^2(I;H^{-1}(\Omega))$, and~$u$ satisfies  $u(T)=S[m(T)]$ in $L^2(\Omega)$ and
\begin{equation}\label{eq:HJB_time_strong_form}
\int_0^T \left[-\dualprod{\partial_t u}{v}{} + (\nu \nabla u,\nabla v)_\Omega + (H[\nabla u],v)_\Omega \right]\mathrm{d}t = \int_0^T \dualprod{F[m]}{v}{} \mathrm{d}t \quad \forall v \in X.
\end{equation}
Both formulations~\eqref{eq:weak-formulation-HJB} and~\eqref{eq:HJB_time_strong_form} of the HJB equation will be used in the following analysis.


\section{Equivalence between errors and residuals}\label{sec:equivalances}
Our first contribution toward deriving computable \emph{a posteriori} error bounds is to prove, in Theorem~\ref{theorem:equivalence-result} below, the equivalence between the norms of the difference between the true solution $(u,m)$ and some general $(\uBar,\mBar)$ and the norms of the residual operators of the MFG system. 

We now define the residual operators associated to the HJB and KFP equations. Let $R_1^X\colon X\times Y_0 \to Y_0^*$ and $R_2^Y\colon X\times Y_0\to X^*$ be defined by
\begin{subequations}\label{eq:R_i-defs}
    \begin{flalign}
        \dualprod{R_1^X(\uBar,\mBar)}{v}{Y^*_0 \times Y_0} &\coloneqq \int_0^T \dualprod{F[\mBar]}{v}{} \dt+ \LTwoInner{S[\mBar(T)]}{v(T)}{\Omega}  \label{eq:R_1^X-def}\\
        & \hspace{0.8cm}-\int_0^T \left[ \dualprod{\partial_t v}{\uBar}{} + \LTwoInner{\nu \nabla \uBar}{\nabla v}{\Omega} + \LTwoInner{H[\nabla \uBar]}{v}{\Omega} \right]\dt, \notag\\
        \dualprod{R_2^Y(\uBar,\mBar)}{w}{X^* \times X} &\coloneqq \int_0^T \left[ \dualprod{G-\partial_t \mBar}{w}{} -\LTwoInner{\nu \nabla \mBar + \mBar \dpH[\nabla \uBar]}{\nabla w}{\Omega} \right]\dt, \label{eq:R_2^Y-def}
    \end{flalign}
\end{subequations}
for all $(v,w) \in Y_0 \times X$. 
Let  $\calR(\cdot,\cdot):X \times Y \to \R_{\geq 0}$ denote the total residual norm functional defined by
\begin{equation} \label{eq:calR-def}
    \calR(\uBar,\mBar) \coloneqq\RH+\RK + \norm{\overline{m}(0)-m_0}_\Omega,
\end{equation}
for any pair $(\uBar,\mBar) \in X \times Y$. It is clear from the definitions of $\calR(\cdot,\cdot)$ and the weak formulation in~\eqref{eq:weak-formulation-XY} that the weak solution $(u,m) \in X \times Y$ satisfies $\calR(u,m) = 0$.
Now we present the main result of this section. 

\begin{theorem}[Equivalence of norms and residuals]\label{theorem:equivalence-result}
    For any $(\overline{u},\overline{m}) \in X \times Y$ such that $\overline{m} \geq 0$ a.e. in $Q_T$, it holds that
    \begin{equation}
        \norm{u-\overline{u}}_X + \norm{m-\overline{m}}_Y \eqsim \calR(\overline{u},\overline{m}), \label{eq:equivalence}
    \end{equation}
    where the hidden constants depend only on $\nu$, $L_H$, $L_{H_p}$, $L_F$, $L_S$, $M_\infty$, $T$, and $\diam \Omega$.
\end{theorem}
The proof is deferred to Section~\ref{sec:Equivalence-proof} below.

\begin{remark}[Equivalence in alternative norms]
     Although not detailed here, we can follow a similar approach to the analysis of this section to prove an equivalence result for other choices of norms on the solution and the residuals. For instance, we can show that
    \begin{equation}\label{eq:alternative-equivalence}
        \norm{u-\overline{u}}_Y + \norm{m-\overline{m}}_Y \eqsim \widetilde{\calR}(\overline{u},\overline{m}), 
    \end{equation}
for any $(\overline{u},\overline{m}) \in Y \times Y$ with $\overline{m}\geq 0$ a.e. in $Q_T$. In this case, the alternative residual norm $\widetilde{\calR}(\overline{u},\overline{m})$ is defined by $ \widetilde{\calR}(\uBar,\mBar) \coloneqq \sum_{i=1}^2 \norm{R_i^Y(\uBar,\mBar)}_{X^*} + \norm{S[\mBar(T)]-\uBar(T)}_\Omega + \norm{\mBar(0)-m_0}_\Omega$, and $R_1^Y(\overline{u},\overline{m})$ is the residual of the HJB equation where the temporal derivative is cast onto $\uBar$ instead of the test function.
However, the equivalence result of Theorem~\ref{theorem:equivalence-result} is better suited for the \emph{a posteriori} error bounds derived below, since it avoids some additional technicalities that appear if one tries to apply~\eqref{eq:alternative-equivalence} to the numerical solution of the FEM scheme.
\end{remark}

\subsection{Stability of the HJB and KFP equations}
Here we establish stability results between the dual norms of the residuals defined in~\eqref{eq:R_i-defs} and the approximation error norms.

\begin{lemma}[Continuity of the residuals]\label{lemma:residual-bounds}
    For all $(\overline{u},\overline{m}) \in X \times Y$, it holds that
    \begin{align}
        \RH  &\lesssim \norm{u-\overline{u}}_X + \norm{m - \overline{m}}_{Y}, \label{eq:R_1-bound} \\
        \RK &\lesssim \norm{u-\overline{u}}_X + \norm{m - \overline{m}}_{Y}, \label{eq:R_2-bound}
    \end{align}
    where the hidden constant of \eqref{eq:R_1-bound} depends on $\nu$, $L_H$, $L_F$, $L_S$, and $\diam \Omega$, while the hidden constant of \eqref{eq:R_2-bound} depends on $\nu$, $L_H$, $L_{H_p}$, $M_\infty$, and $\diam \Omega$. 
\end{lemma}
\begin{proof}
Subtracting the weak formulation~\eqref{eq:weak-formulation-XY} from the residual equations~\eqref{eq:R_i-defs} yields the identities
\begin{subequations} \label{eq:weak-formulation-identity}
\begin{multline}
     \dualprod{R_1^X(\uBar,\mBar)}{v}{Y^*_0 \times Y_0} \\
     = \int_0^T \dualprod{F[\mBar]-F[m]}{v}{}\dt+ \LTwoInner{S[\mBar(T)]-S[m(T)]}{v(T)}{\Omega}  \\
        +\int_0^T \left[ \dualprod{\partial_t v}{u-\uBar}{} + \LTwoInner{\nu \nabla (u-\uBar)}{\nabla v}{\Omega} + \LTwoInner{H[\nabla u]-H[\nabla \uBar]}{v}{\Omega} \right]\dt, \label{eq:weak-formulation-identity-HJB}
\end{multline}
\vspace{-\belowdisplayskip}
\begin{multline}
    \dualprod{R_2^Y(\uBar,\mBar)}{w}{X^* \times X} \\
    = \int_0^T \left[ \dualprod{\partial_t (m-\mBar)}{w}{} + \LTwoInner{\nu \nabla (m-\mBar) }{\nabla w}{\Omega} \right]\dt\\
    +\int_0^T  \LTwoInner{m \dpH[\nabla u]-\mBar \dpH[\nabla \uBar]}{\nabla w}{\Omega}\dt,
    \label{eq:weak-formulation-identity-KFP}
\end{multline}
\end{subequations}
for all $v \in Y_0$ and all $w \in X$.

To prove~\eqref{eq:R_1-bound}, we apply a sequence of triangle and Cauchy--Schwarz inequalities to~\eqref{eq:weak-formulation-identity-HJB}, then we apply the Lipschitz continuity of $H$, $F$, and $S$, see~\eqref{eq:H-cts-condition}, \eqref{eq:F-cts-condition}, and \eqref{eq:S-cts-condition}, followed by embedding $Y \hookrightarrow Z$ and the Poincar\'e inequality.
Similarly, we bound~\eqref{eq:weak-formulation-identity-KFP} by applying a sequence of triangle and Cauchy--Schwarz inequalities, yielding
\begin{equation}
    \RK \lesssim \norm{m-\mBar}_Y + \norm{m \dpH[\nabla u]-\mBar \dpH[\nabla \uBar]}_{L^2(Q_T; \R^d)}. \label{eq:proof:R_2-bound-1}
\end{equation}
To bound the remaining drift term in~\eqref{eq:proof:R_2-bound-1}, we add and subtract terms, apply the triangle inequality, then apply the continuity and boundedness of $\dpH$ and the uniform bound of $m$, see~\eqref{eq:dpH-cts-condition}, \eqref{eq:dpH-bounded-condition}, and \eqref{eq:M_infty}, and the Poincar\'e inequality to obtain
\begin{equation}
    \norm{m \dpH[\nabla u]-\mBar \dpH[\nabla \uBar]}_{L^2(Q_T; \R^d)} \lesssim   \norm{u-\overline{u}}_X +  \norm{m-\overline{m}}_Y, \label{eq:drift-term-bound}
\end{equation}
from which~\eqref{eq:R_2-bound} follows.
\end{proof}

Next, we record some inf-sup stability results that are useful for proving the stability of the HJB and KFP equations.

\begin{lemma}[Norm Stability estimates]\label{lemma:inf-sup-XY}
Let $B(\cdot,\cdot):Y \times X \to \R$ be a bilinear form defined, for any choice of $\overline{b} \in L^\infty(Q_T;\R^d)$, by ~$B(v,w)\coloneqq \int_0^T\dualprod{\partial_t v}{w}{}+\LTwoInner{\nu \nabla v + v \overline{b}}{\nabla w}{\Omega}\dt$ for all $v \in Y$, $w \in X$.
    For every $\phi \in X$ and $\psi \in Y$, it holds that
    \begin{align}
    \norm{\psi}_Y &\lesssim \sup_{w \in X} \frac{B(\psi,w)}{\norm{w}_X} + \norm{\psi(0)}_\Omega \label{eq:inf-sup-Y-L}, \\
    \norm{\phi}_X &\lesssim \sup_{v \in Y_0} \frac{B(v,\phi)}{\norm{v}_Y} \label{eq:inf-sup-X-L},
    \end{align}
    where the hidden constants depend only on $\nu$, the $L^{\infty}(Q_T;\R^d)$-norm of $\overline{b}$, $T$, and $\diam \Omega$.
\end{lemma}

We omit the proof of Lemma~\ref{lemma:inf-sup-XY}, since~\eqref{eq:inf-sup-Y-L} is easily deduced from the inf-sup stability of the heat equation, see e.g.~\cite[Theorem~2.1]{ErnSmearsVohralik2017b}, and from Gronwall's inequality, while~\eqref{eq:inf-sup-X-L} follows from~\eqref{eq:inf-sup-Y-L} by duality.
The next Lemma shows the stability of the HJB and KFP when considered separately.

\begin{lemma}[Stability of the HJB and KFP equations] \label{lemma:cts-approx-stab}
For any $(\overline{u},\overline{m}) \in X \times Y$, it holds that
\begin{align}
    \norm{u-\overline{u}}_X &\lesssim \RH + \norm{m-\overline{m}}_{Y}, \label{eq:u-error-stab} \\
    \norm{m-\overline{m}}_Y &\lesssim  \norm{R_2^Y(\overline{u},\overline{m})}_{X^*} + \norm{m \left( \dpH[\nabla \uBar]-\dpH[\nabla u] \right)}_{L^2(Q_T;\R^d)} + \norm{m_0-\overline{m}(0)}_{\Omega}, \label{eq:m-error-stab}
\end{align}
    where the hidden constant of~\eqref{eq:u-error-stab} depends only on $\nu$, $L_H$, $L_F$, $L_S$, $T$, and $\diam \Omega$, and the hidden constant of~\eqref{eq:m-error-stab} depends only on $\nu$, $L_H$, $T$, and $\diam \Omega$.
\end{lemma}
\begin{proof}
Let $\overline{b}_1 \in L^{\infty}(Q_T;\R^d)$ be defined by $\overline{b}_1 \coloneqq -\int_0^1 \dpH[\nabla \uBar + s(\nabla u - \nabla \uBar)]~ds$ and let  $B_1(\cdot,\cdot)$ denote the bilinear form of Lemma~\ref{lemma:inf-sup-XY} corresponding to the choice of $b=\overline{b}_1$. Notice that
     \begin{equation}
        H[\nabla \overline{u}]- H[\nabla u] + \overline{b}_1 \cdot \nabla (\overline{u}-u) = 0, \quad \text{a.e. in } Q_T, \quad \norm{\overline{b}_1}_{L^\infty(Q_T;\R^d)} \leq L_H. \label{eq:proof:mean-value-H}
    \end{equation}
Using~\eqref{eq:weak-formulation-identity-HJB} and \eqref{eq:proof:mean-value-H}, we obtain the identity
\begin{multline}
    B_1(v,u-\uBar) = \dualprod{R_1^X(\uBar,\mBar)}{v}{Y_0^* \times Y_0} 
    \\ + \int_0^T \dualprod{F[m]-F[\mBar]}{v}{}\dt + \LTwoInner{S[m(T)]-S[\mBar(T)]}{v(T)}{\Omega} \quad \forall v \in Y_0. \label{eq:proof:u-error-bound-1}
\end{multline}
Applying triangle and Cauchy--Schwarz inequalities, the continuity conditions of $F$ and $S$, and the embedding $Y \hookrightarrow Z$ to~\eqref{eq:proof:u-error-bound-1} then substitution into the bound~\eqref{eq:inf-sup-X-L}, recalling the definition of the $Y$-norm in~\eqref{eq:Y-norm-def}, yields~\eqref{eq:u-error-stab}.

We now prove~\eqref{eq:m-error-stab}. Let $\overline{b}_2\coloneqq \dpH[\nabla \uBar]$ and $B_2(\cdot,\cdot)$ denote the corresponding bilinear form of Lemma~\ref{lemma:inf-sup-XY} where $b=\overline{b}_2$. We similarly recover the KFP residual from the upper bound term of~\eqref{eq:inf-sup-Y-L} using~\eqref{eq:weak-formulation-identity-KFP}
\begin{equation}
    B_2(m-\mBar,w) = \dualprod{R_2^Y(\uBar,\mBar)}{w}{X^* \times X} 
    + \int_0^T \LTwoInner{m \left( \dpH[\nabla \uBar]-\dpH[\nabla u] \right)}{\nabla w}{\Omega}\dt,\label{eq:proof:m-error-bound-1}
\end{equation}
for all $w \in X$.
Substitution of~\eqref{eq:proof:m-error-bound-1} into~\eqref{eq:inf-sup-Y-L}, followed by applications of triangle and Cauchy--Schwarz inequalities leads to~\eqref{eq:m-error-stab}.
\end{proof}

As will become clear in the proof of Theorem~\ref{theorem:equivalence-result}, it is important that we do not further bound the second term in the upper bound of~\eqref{eq:m-error-stab} using the continuity of $\dpH$ and the boundedness of $m$. 

\subsection{Proof of the equivalence result}\label{sec:Equivalence-proof}
\begin{proof}[Proof of Theorem~\ref{theorem:equivalence-result}]
   
    A direct application of Lemma~\ref{lemma:residual-bounds} to~\eqref{eq:calR-def}, followed by the embedding~\eqref{eq:Y-embedding}, provides
    \begin{equation*}
        \calR(\uBar,\mBar) \lesssim \norm{u-\uBar}_X + \norm{m-\mBar}_Y + \norm{\mBar(0)-m_0}_\Omega   \lesssim \norm{u-\uBar}_X + \norm{m-\mBar}_Y,
    \end{equation*}
    for any pair $(\uBar,\mBar) \in X \times Y$.
    This provides the lower bound claimed in~\eqref{eq:equivalence}.
    Adding the stability bounds of~\eqref{eq:u-error-stab} and~\eqref{eq:m-error-stab} provides
    \begin{equation*}
        \norm{u-\uBar}_X + \norm{m-\mBar}_Y \lesssim \calR(\uBar,\mBar)+ \norm{m-\overline{m}}_{Y}+ \norm{m \left( \dpH[\nabla \uBar]-\dpH[\nabla u] \right)}_{L^2(Q_T;\R^d)}.
    \end{equation*}
    Next we apply again the stability of the KFP equation~\eqref{eq:m-error-stab} and bound the resulting terms by $\calR(\uBar,\mBar)$ to get
    \begin{equation}
        \norm{u-\uBar}_X + \norm{m-\mBar}_Y \lesssim \calR(\uBar,\mBar)
        + \norm{m \left( \dpH[\nabla \uBar]-\dpH[\nabla u] \right)}_{L^2(Q_T;\R^d)}, \label{eq:proof:drift-error-bound}
    \end{equation}
    for any pair $(\uBar,\mBar) \in X \times Y$. 
    We now show how to bound the last term in the right-hand side of~\eqref{eq:proof:drift-error-bound}.
    Let $\xi \in Y$ be defined as the weak solution of  $\int_0^T \dualprod{\partial_t \xi}{w}{} + \LTwoInner{\nu \nabla \xi}{\nabla w}{\Omega} =0$ for all $w \in X$, with initial condition $\xi(0)=m_0 - \mBar(0)$ in $\Omega$.
    Note that ~$\norm{\xi}_Y=\norm{m_0-\mBar(0)}_\Omega$ (cf. \cite[Theorem~2.1]{ErnSmearsVohralik2017b}), and thus $\norm{\xi}_Y \leq \calR(\uBar,\mBar)$ by definition of $\calR(\cdot,\cdot)$. 
    Let $\mBar_\xi \coloneqq \mBar + \xi$. We test~\eqref{eq:weak-formulation-identity-HJB} with $v=\mBar_\xi-m \in Y_0$, and~\eqref{eq:weak-formulation-identity-KFP} with $w=\uBar-u \in X$, then subtract the resulting equations to obtain
    \begin{equation*}
        \dualprod{R_1^X(\uBar,\mBar)}{\mBar_\xi-m}{Y_0^* \times Y_0} - \dualprod{R_2(\uBar,\mBar)}{\uBar-u}{X^* \times X} = I_1+I_2 +I_3,
    \end{equation*}
    where the terms on the right-hand side above are defined by
    \begin{align*}
        I_1 &\coloneqq  \int_0^T \dualprod{F[\mBar]-F[m]}{\mBar-m}{}\dt+ \LTwoInner{S[\mBar(T)]-S[m(T)]}{\mBar(T)-m(T)}{\Omega}, \\
        I_2 &\coloneqq \int_0^T \LTwoInner{m}{D_H[\nabla \uBar,\nabla u]}{\Omega} + \LTwoInner{\mBar}{D_H[\nabla u,\nabla \uBar]}{\Omega}\dt, \\
        I_3 &\coloneqq  \int_0^T \dualprod{F[\mBar]-F[m]}{\xi}{}\dt+ \LTwoInner{S[\mBar(T)]-S[m(T)]}{\xi(T)}{\Omega} \\
        &\qquad + \int_0^T \LTwoInner{H[\nabla u]-H[\nabla \uBar]}{\xi}{\Omega}\dt.
    \end{align*} 
    Recall that the Bregman divergence terms $D_H$ in the above expansion are defined in~\eqref{eq:Bregman-div-def}.
    Note that the monotonicity condition in~\eqref{eq:F&S-monotonicity-condition} implies that $I_1 \geq 0$.
    Also note that the nonnegativity of $m$, $\mBar$, and of the Bregman divergence terms, imply that $I_2 \geq \int_0^T \LTwoInner{m}{D_H[\nabla \uBar,\nabla u]}{\Omega}\dt$.
    Therefore, we apply the $L^\infty$-bounds of~\eqref{eq:Bregman-inequality} and~\eqref{eq:M_infty} to obtain a lower bound on $I_2$
    \begin{align*}
         \norm{m \left( \dpH[\nabla \uBar]-\dpH[\nabla u] \right)}_{L^2(Q_T;\R^d)}^2 &\leq M_\infty \int_{Q_T} m \abs{\dpH[\nabla \uBar]-\dpH[\nabla u] }^2\dx\dt \nonumber \\
         &\leq 2 L_{H_p} M_\infty\int_0^T \LTwoInner{m}{D_H[\nabla \uBar,\nabla u]}{\Omega}\dt \lesssim I_2.
    \end{align*}
    Lipschitz continuity of~$H$,~$F$, and~$S$, and the embedding~\eqref{eq:Y-embedding} imply that
    \begin{equation*}
        \abs{I_3} \lesssim (\norm{m-\mBar}_Y + \norm{u-\uBar}_X) \norm{\xi}_Y \lesssim (\norm{m-\mBar}_Y + \norm{u-\uBar}_X) \calR(\uBar,\mBar),
    \end{equation*}
    where we have used the bounds of $\xi$ as shown above.
    Note that the definition of $\mBar_\xi$ and $\calR(\uBar,\mBar)$ imply that
    \begin{align*}
         \dualprod{R_1^X(\uBar,\mBar)}{\mBar_\xi-m}{Y_0^* \times Y_0} &\leq \RH \norm{\mBar_\xi-m}_Y \\
         &\leq \RH (\norm{\mBar-m}_Y + \norm{\xi}_Y) \\
         &\leq \calR(\uBar,\mBar) \norm{\mBar-m}_Y +  [\calR(\uBar,\mBar)]^2.
    \end{align*}
    Therefore, we deduce the bound
    \begin{equation*}
    \norm{m \left( \dpH[\nabla \uBar]-\dpH[\nabla u] \right)}_{L^2(Q_T;\R^d)}^2 \lesssim [\calR(\uBar,\mBar)]^2 +  \calR(\uBar,\mBar) \left(  \norm{u-\uBar}_X +  \norm{m-\mBar}_Y \right). \label{eq:proof:drift-error-bound-2}
\end{equation*}
    Then, we use~\eqref{eq:proof:drift-error-bound} and the above inequality to get 
    \begin{align*}
        \norm{u-\uBar}_X + \norm{m-\mBar}_Y&\lesssim \calR(\uBar,\mBar) + [\calR(\uBar,\mBar)]^{\frac{1}{2}} \left( \norm{u-\uBar}_X + \norm{m-\mBar}_Y \right)^{\frac{1}{2}}.
    \end{align*}
    We then apply Young's inequality (with a suitable parameter) on the final term of the right-hand side above to deduce that $ \norm{u-\uBar}_X + \norm{m-\mBar}_Y \lesssim \calR(\uBar,\mBar)$, thus completing the proof of Theorem~\ref{theorem:equivalence-result}.
    \end{proof}
 

\section{Stabilized finite element methods}\label{sec:FEM-scheme}
In this section, we present the numerical framework of the stabilized FEM used to approximate the MFG system \eqref{eq:MFG-system}.

\subsection{Spatial discretization}
To avoid unnecessary technicalities, we shall only consider a fixed, time-independent spatial discretization of $\Omega$. 
Let $\mathcal{T}$ denote a conforming simplicial mesh of the domain $\Omega$. 
We adopt the convention that each element $K\in\T$ is closed.
For each element $K\in\T$, we let $h_K$ denote the diameter of $K$ and $\varrho_K$ denote the largest diameter of an inscribed ball in $K$. 
Let $\theta_\T\coloneqq \max_{K\in\T} h_K/\rho_K$ denote the shape-regularity parameter of the mesh~$\T$.
Let $\calV$ denote the set of all vertices of $\T$ and let $\calV^I=\calV \cap \Omega$ denote the set of interior vertices of $\T$.
Two distinct vertices are called neighbours if they belong to a common element of $\T$.
Let $\calE$ denote the set of edges of the mesh $\T$, i.e.\ the set of all closed line segments formed by all pairs of neighbouring vertices.
Given a simplex $K\in\mathcal{T}$ we denote the set of edges contained in $K$ by $\calE_K$. 
Let $\calF$ denote the set of all faces of $\T$, and let $\calF^I$ denote the subset of all interior faces of $\calF$, i.e.\ all faces that are not contained in $\partial\Omega$.
Note that for $d=2$, edges and faces of the mesh coincide.
For each face $F\in\calF$, let $h_F$ denote the diameter of $F$.
For each element $K\in\T$, we let $\calF^I_K$ denote the set of interior faces $F\in\calF^I$ that are contained in $K$.
For each face $F\in\calF$, we let $\TF$ denote the set of elements that contain $F$, and we let $\omega_F \coloneqq \bigcup_{K\in\TF}K$ denote its associated patch.
For each element $K\in\T$, we define the set $\tTK$ of face-sharing neighbouring elements and the associated patch $\omega_K$
\begin{equation}
\begin{aligned}
\tTK \coloneqq \bigcup_{F\in\calFKi}\TF, &&& \omega_K\coloneqq \bigcup_{K^\prime\in\tTK} K^\prime = \bigcup_{F\in\calFKi}\omega_F.
\end{aligned}
\end{equation}
Note that $K\in\tTK$ for all $K\in\T$.

Let $V_h \subset H^1_0(\Omega)$ denote the linear Lagrange finite element space constructed over $\T$ defined by
\begin{equation}
    V_h \coloneqq \{ v_h \in H^1_0(\Omega)~ :~ v_h |_{K} \in \mathcal{P}_1(K) \quad \forall K \in \T \}, \label{eq:fem-space-def}
\end{equation}
where $\mathcal{P}_1(K)$ denotes the space of linear polynomials on $K$. 
Let $\{\varphi_z \}_{z \in \calV^I}$ denote the set of standard Lagrangian basis functions characterised by $\varphi_z(z) = 1$ and $\varphi_z(z')=0$ for all $z' \in \calV \backslash \{ z \}$.

\subsection{Space-time finite element spaces}
Let $\mathcal{J} \coloneqq \{ I_n \}_{n=1}^N$ denote a partition of $\overline{I} = [0,T]$ into intervals $I_n \coloneqq (t_{n-1},t_n)$ with $0=t_0 \leq t_{n-1} < t_n \leq t_N$ for each $n \in \{1,...,N\}$. For each $n \in \{1,...,N\}$, the time-step size is denoted by $\tau_n \coloneqq t_n - t_{n-1}$.
Let $\V$ denote the space of $V_h$-valued functions
defined on $[0,T]$ that are piecewise constant with respect to $\mathcal{J}$. In other words, $\V$ is defined by
\begin{equation}
    \V \coloneqq \left\{ \vht \in X\,:\, \vht |_{I_n} \in V_h \text{ is constant }  \forall n \in \{1,...,N \} \right\}.
\end{equation}
The implicit Euler time stepping method naturally leads to numerical solutions that can be viewed as pairs of functions in $\V$.
However, given the forward-backward structure of the equations of the MFG system with initial and final time conditions on the different solution components, it will be natural to view the numerical solutions as functions that are defined everywhere on $[0,T]$, with different  left- and right- continuity properties between the value function and density, owing to the different time directions of the equations.
For a function $\vht \in \V$, let $\vht(t_n^-)$ and $\vht(t_n^+)$ denote respectively the left- and right-limits of $\vht$ at time $t_n$. Also note that $\vht(t_n^+) = \vht|_{I_{n+1}}$ and $\vht(t_n^-) = \vht|_{I_n}$ for all $n \in \{0,...,N-1\}$, and $\vht(0^+) = \vht|_{I_1}$ and  $\vht(T^-) = \vht|_{I_N}$.
We say that $\vht$ is left-continuous if $\vht(t_n) = \vht(t^-_n)$ for all $n \in \{1,...,N \}$ and right-continuous if $\vht(t_n) = \vht(t^+_n)$ for all $n \in \{0,...,N-1 \}$. 
We define the right- and left- continuous finite element spaces by
\begin{equation}\label{eq:signed_FEM_spaces}
    \begin{aligned}
        \Vone &\coloneqq \{\vht:[0,T] \to V_h,~ \vht|_{(0,T)} \in \V ~:~   \vht \text{ is right-continuous} \}, \\
        \Vtwo &\coloneqq \{ \vht:[0,T] \to V_h,~ \vht|_{(0,T)} \in \V ~:~  \vht \text{ is left-continuous} \}.
    \end{aligned}
\end{equation}
We emphasize that functions in $\Vi$ are considered to be defined at all points in $[0,T]$, and two functions that agree on up to a null set of $[0,T]$ are not identified. We take this viewpoint in order to handle the initial and final time conditions of the problem in a natural manner.
In order to alleviate the notation, the dependence of the spaces on the mesh~$\T$ and time partition~$\mathcal{J}$ will be left implicit.
Note that the spaces~$\Vi$ defined above are only generally conforming with respect to the space~$X$, and not~$Y$, since the spaces $\Vi$ contain functions that may be discontinuous at the timestep points $t_n \in (0,T)$. 
Therefore, in order to consider the error between the exact and discrete solutions, we consider the sum space~$\Vi + Y$ defined by
\begin{equation} \label{eq:space-time-fem-space-defs}
    \Vi+Y \coloneqq \left\{ w \in X~:~ \exists (\vht,v) \in \Vi \times Y~\mathrm{s.t.}~ w = \vht + v \right\}.
\end{equation}
Note that the spaces $\Vi +Y$ are not generally direct sum spaces.
For a given function $w \in \Vi +Y$ and time level $t_n \in (0,T)$, the temporal jump operators are defined by
\begin{equation}
    \timejump{w}_n \coloneqq w(t_n^-) - w(t_n^+) \quad \forall n\in \{1,\ldots,N-1\}.
\end{equation}
Moreover, the jumps at the initial and final times are defined by $\timejump{w}_0 = w(0) - w(0^+)$ and $\timejump{w}_N = w(T^-) - w(T)$ respectively.
The temporal reconstruction operators $\calI_i : \Vi + Y \to Y$ are defined by
\begin{equation} \label{eq:Reconstruction-defs}
    \begin{aligned}
        (\calI_1 w)(t) &\coloneqq w(t) - \frac{t-t_{n-1}}{\taun} \timejump{w}_n,~~t \in [t_{n-1},t_n),~ n \in \{1,...,N\},~ w \in \Vone +Y, \\
        (\calI_2 w)(t) &\coloneqq w(t) + \frac{t_n-t}{\taun} \timejump{w}_{n-1}, ~~ t \in (t_{n-1},t_n],~ n \in \{1,...,N\},~ w \in \Vtwo +Y.
    \end{aligned}  
\end{equation}
It is easy to check that the operators $\calI_i$ are linear, and moreover take values in $Y$ as a result of the fact that they map into spaces of functions that are continuous and piecewise smooth in time.
Also, by the embedding of $Y \hookrightarrow C([0,T];L^2(\Omega))$, it is clear that $\calI_i w = w$ for any $w \in Y$, since the vanishing jumps $\timejump{w}_n=0\in H^1_0(\Omega)$ vanish for any $w\in Y$.
We equip the spaces $\Vi+Y$ with an extended $Y$-norm defined by
\begin{equation} \label{eq:error-norm-defs}
\begin{aligned}
    \norm{v}_{\Vi+Y} \coloneqq \norm{v-\calI_i v}_X + \norm{\calI_i v}_Y, \quad \forall v \in \Vi + Y.
\end{aligned}
\end{equation}
The norm defined above is called the extended $Y$-norm since it defines a consistent extension of the norm $\norm{\cdot}_Y$ of~\eqref{eq:Y-norm-def} from the space~$Y$ to the sum space~$\Vi+Y$, which is to say that $\norm{v}_{\Vi+Y} = \norm{v}_Y$ for any $v \in Y$.

\subsection{Discretized problem}
As explained in the introduction, we consider here a family of stabilized FEM, which includes for instance the method in~\cite{OsborneSmears2025i}, based on an implicit Euler discretization in time, and a piecewise affine FEM is used in space. The stabilization is motivated by the need to preserve nonnegativity of the approximate densities, which plays an important role in the uniqueness of numerical solutions and also in the nonnegativity condition appearing in Theorem~\ref{theorem:equivalence-result}.
We now consider abstract stabilization terms of the form $\calS_i(\cdot;\cdot; \cdot): \Vone \times \Vtwo \times \V \to \R$, $i \in \{ 1,2\}$ that are allowed to be nonlinear with respect to first two arguments, but are assumed to be linear with respect to the third argument. 
Example~\ref{ex:stabilizations} below details some concrete  stabilizations that we have in mind.

The class of FEM that we consider are of the form: find $(\uht,\mht)\in \Vone\times \Vtwo$ such that
\begin{subequations}\label{eq:FEM-scheme-def}
\begin{flalign}
    &\int_0^T \LTwoInner{-\partial_t \Uht}{\vht}{\Omega} + \LTwoInner{\nu \nabla \uht}{\nabla \vht}{\Omega} + \LTwoInner{H[\nabla \uht]}{\vht}{\Omega}\dt  \\
     &\hspace{3.1cm}  + \calS_1(\uht;\mht;\vht) = \int_0^T \dualprod{F[\mht]}{\vht}{} \dt \notag 
      \\
    &\int_0^T \LTwoInner{\partial_t \Mht}{\vht}{\Omega} + \LTwoInner{\nu \nabla \mht}{\nabla \vht}{\Omega} + \LTwoInner{\mht \dpH[\nabla \uht]}{\nabla \vht}{\Omega}\dt \\
 & \hspace{3.1cm}+ \calS_2(\uht;\mht;\vht) = \int_0^T \dualprod{G}{\vht}{} \dt, \notag 
\end{flalign}
\end{subequations}
for all $\vht\in \V$, where we use the shorthand notation $\Uht \coloneqq \calI_1 \uht$ and $\Mht\coloneqq \calI_2 \mht$, and such that $\mht(0) = \Pi_0 m_0$ and $\uht(T) = \Pi_T S[\mht(T)]$, where~$\Pi_0$ and~$\Pi_T$ are some quasi-interpolation operators from~$L^2(\Omega)$ to the finite element space~$V_h$ that are chosen by the user to approximate the initial and final time conditions.
In practice, the stabilizations $\mathcal{S}_i$ are chosen to ensure good properties of the FEM~\eqref{eq:FEM-scheme-def}, such as existence and uniqueness of the numerical solution, and stability, see~\cite{OsborneSmears2025i} for further details.
However, the only assumptions that we require for the general \emph{a posteriori} error analysis of Section~\ref{sec:a-post-bounds} below is that a solution of~\eqref{eq:FEM-scheme-def} exists, and that discrete densities are nonnegative:\smallskip
\begin{enumerate}[label={(H\arabic*)},resume]
\item \label{H:stabilization_main} if $(\uht,\mht)\in \Vone\times \Vtwo$ is a solution of~\eqref{eq:FEM-scheme-def} then $\mht \geq 0$ in~$Q_T$.
\end{enumerate}\smallskip 

\noindent Observe that if $\mht$ is nonnegative in~$Q_T$ then so is $\Mht\coloneqq \calI_2\mht $.
\begin{example}[Examples of stabilizations]\label{ex:stabilizations}
An example of a choice of stabilizations $\mathcal{S}_i$ that satisfies Hypothesis~\ref{H:stabilization_main} can be found in~\cite{OsborneSmears2025i}.
There, the stabilization is based on mass lumping of the $L^2$ inner-product for the temporal derivative terms and linear stabilization of spatial first-order derivative terms, resulting in a discrete maximum principle (DMP) for the discrete scheme.
In particular, the stabilizations in~\cite{OsborneSmears2025i} are of the form
\begin{equation}\label{eq:Stab-example-defs}
\begin{aligned}
    \calS_1(\uht;\mht;\vht) &\coloneqq \int_0^T \LTwoInner{\partial_t \Uht}{\vht}{\Omega} - \LTwoInner{\partial_t \Uht}{\vht}{\Omega,h} + (D_h \nabla \uht, \nabla \vht)_\Omega\dt,\\
    \calS_2(\uht;\mht;\vht) &\coloneqq \int_0^T \LTwoInner{\partial_t \Mht}{\wht}{\Omega,h} - \LTwoInner{\partial_t \Mht}{\wht}{\Omega} + (D_h \nabla \mht, \nabla \wht)_\Omega\dt,
\end{aligned}
\end{equation}
where $(\cdot,\cdot)_{\Omega,h} : V_h \times V_h \to \R$ denotes the mass-lumped inner-product defined by
\begin{equation}\label{eq:mass-lump-def}
    \massLumped{v_h}{w_h} \coloneqq \sum_{z \in \calV^I} \LTwoInner{\varphi_z}{1}{\Omega} v_h(z)~ w_h(z) \quad \forall v_h, w_h \in V_h,
\end{equation}
and where $D_h \in L^\infty(\Omega; \R^{d\times d}_{\mathrm{sym}})$ is a piecewise constant symmetric matrix-valued function of the form $D_h |_K = \sum_{E \in \calE_K} \gamma_E t_E \otimes t_E$ for each element $ K \in \T $, where $\calE_K$ denotes the set of edges of $K$ (i.e.\ line segments between vertices), and $t_E$ is a chosen unit tangent vector to edge $E$, and $\gamma_E \geq 0$ is a suitably chosen weight.
It was shown in~\cite[Theorem~4.2]{OsborneSmears2024i} that, under a suitable condition on the meshes and for weights $\gamma_E$ chosen to be of the same order as the mesh-size, the stabilized schemes satisfy Hypothesis~\ref{H:stabilization_main}.
\end{example}

\section{A posteriori error bounds}\label{sec:a-post-bounds}
In this section we present the general analysis of \emph{a posteriori} error estimators for the class of FEM schemes defined above.
For the sake of brevity of exposition, we shall assume in the following that the coupling term $F$ has images in $L^2(Q_T)$ for all arguments in $X$, and that $G\in L^2(Q_T)$.
We refer the reader to~\cite{ErnSmearsVohralik2017,KreuzerVeeser2021} and references therein for further details on the treatment of PDE with $H^{-1}$ source terms.

\subsection{The estimators}\label{sec:estimator-defs}
We start by defining the estimators that appear in the error bound.
The \emph{a posteriori} estimator $\eta(\uht,\mht)$ is defined by
\begin{equation}\label{eq:eta-total-def}
    \eta(\uht,\mht) \coloneqq \sum_{i=1}^2 \left[ \etaRecon{i} + \etaRes{i} +\etaStab{i} \right]+ \eta_0 + \eta_T,
\end{equation}
where the $\etaRecon{i}$ denote the temporal jump estimators, the~$\etaRes{i}$ denote the PDE residual estimators, the~$\etaStab{i}$ denote the stabilization estimators, and $\eta_0$ and $\eta_T$ denote the initial and final time condition estimators.
These terms are defined respectively in \eqref{eq:etaRecon-global-defs}, \eqref{eq:eta-res-defs}, \eqref{eq:etaStab-defs} and \eqref{eq:eta0-T-defs} below.
To alleviate the notation, the dependency of the above estimators on the numerical solution is left implicit.

\paragraph{The temporal jump estimators}
The temporal jump estimators $\etaRecon{i}$ measure the lack of conformity in the space $Y$ of the functions $\uht$ and $\mht$, which are generally discontinuous across time-intervals.
Let $\etaRecon{i}$, $i\in \{1,2\}$ be defined by
\begin{equation}
    \etaRecon{i}^2 \coloneqq \sum_{n=1}^N \sum_{K\in \T} \etaRecon{K,n,i}^2, 
    \quad i \in \{1,2\}. \label{eq:etaRecon-global-defs}
\end{equation}
where the local contributions are given by $\etaRecon{K,n,1} \coloneqq \norm{\nabla (\uht-\Uht)}_{L^2(I_n \times K)}$ and $\etaRecon{K,n,2} \coloneqq \norm{\nabla(\mht-\Mht)}_{L^2(I_n \times K)}$ for each time-interval $I_n$, $n\in\{1,\dots,N\}$, and each element $K\in \T$.
The temporal jump estimators can be evaluated straightforwardly through the formulas
\begin{equation}
\begin{aligned}
 [\etaRecon{1}]^2=\frac{1}{3}\sum_{n=1}^N \tau_n \norm{\nabla \timejump{\uht}_n  }_\Omega^2, &&&
 [\etaRecon{2}]^2=\frac{1}{3}\sum_{n=1}^N \tau_n \big\lVert \nabla \timejump{\mht}_{n-1}  \big\rVert_\Omega^2 .
 \end{aligned}
\end{equation}

\paragraph{The residual estimators}
For each time-interval~$I_n$, $n\in\{1,\dots,N\}$, and element~$K\in \T$, we define the local volume residuals by
\begin{subequations}\label{eq:volume-residual-defs}
\begin{align}
r_{K,n,1} &\coloneqq  F[\mht] + \partial_t \Uht + \nu \Delta (\uht|_K) -  H[\nabla \uht], \\ 
r_{K,n,2} &\coloneqq  G - \partial_t  \Mht  + \nu \Delta (\mht|_K) +  \mathrm{div}(\mht  \dpH[\nabla \uht])  .
\end{align}
\end{subequations}
Since the functions $\uht$ and $\mht$ are piecewise affine on each element of the mesh, the Laplacian terms in the volume residuals above vanish; however we choose to write them explicitly to emphasize the relation of the volume residuals to the strong form of the original PDE.
Also, for each interior face $F \in \calF^I$, we define the local spatial jump residuals by
\begin{subequations} \label{eq:spatial-jump-residual-defs}
    \begin{align}
         j_{F,n,1} &\coloneqq \nu \jump{\nabla \uht \cdot n_F}_{F}, \label{eq:spatial-jump-residual-hjb-def}\\
        j_{F,n,2} &\coloneqq \nu \jump{\nabla \mht \cdot n_F}_{F} +  \mht \jump{\dpH[\nabla \uht] \cdot n_F}_{F}. \label{eq:spatial-jump-residual-kfp-def}
    \end{align}
\end{subequations}
For each $i\in\{1,2\}$, the total residual estimator $\etaRes{i}$ is defined by
\begin{gather} 
    \etaRes{i}^2 \coloneqq \sum_{n=1}^N \sum_{K \in \T} \etaRes{K,n,i}^2, \label{eq:eta-res-defs}
    \\
   \etaRes{K,n,i}^2 \coloneqq h_K^2 \norm{r_{K,n,i}}^2_{L^2(I_n \times K)} + \sum_{F \in \calF_K^I} h_F \norm{j_{F,n,i}}_{L^2(I_n \times F)}^2. \label{eq:eta-res-local-defs}
\end{gather}

\paragraph{The stabilization estimators}
The stabilization estimators are defined by
\begin{align}
    \etaStab{i} \coloneqq \sup_{\vht \in\V\setminus\{0\}} \frac{\calS_i(\uht;\mht;\vht)}{\norm{\vht}_X} \quad \forall i \in \{ 1,2 \}.\label{eq:etaStab-defs}
\end{align}
Since the stabilizations are linear in the third argument, the stabilization estimators are computable in practice by solving a discrete linear elliptic problem in the finite element space $V_h$ on each time-interval. We refer the reader also to~\cite[Remark 5.2]{OsborneSmearsWells2025} for further discussion of how the computations can be made more efficient in practice by approximating the stabilization estimators via standard preconditioners for elliptic problems.
However, for the general class of stabilizations above, it does not appear to be possible to localize the stabilization estimators across the spatial mesh in general. 
This motivates the further analysis in Section~\ref{sec:local-estimators}, where it will be seen that the computation of the stabilization estimators can be avoided in practice if the stabilizations have some additional structure.

\paragraph{Initial and final time condition estimators}
Recall that the FEM scheme~\eqref{eq:FEM-scheme-def} imposes the initial condition $\mht(0)=\Pi_0 m_0$ and final time condition $\uht(T)=\Pi_T S[\mht(T)]$ for some user-chosen operators $\Pi_0$ and $\Pi_T$. 
The initial- and final-time condition estimators are defined by
\begin{equation}
    \eta_0 \coloneqq \norm{m_0 -\Pi_0 m_0}_\Omega, \quad \eta_T \coloneqq \norm{S[\mht(T)]-\Pi_TS[\mht(T)]}_\Omega. \label{eq:eta0-T-defs}
\end{equation}
One could alternatively consider the terms $\eta_0$ and $\eta_T$ as data oscillation terms. There is no practical difference in the different point of views, yet we choose to consider these terms as estimators since one can show their efficiency properties, see Remark~\ref{rem:efficiency_initial_final} and Theorem~\ref{theorem:global-efficiency} below.

\paragraph{Data oscillation terms}
Finally, we define the temporal and spatial residual estimators that will appear below in Theorems~\ref{theorem:reliability} and Theorems~\ref{theorem:local-efficiency} and~\ref{theorem:global-efficiency} respectively. 
We note that it is not generally possible to avoid the dependence of the data oscillation on the numerical solution, owing to the nonlinearity of the problem.
To define the temporal data oscillation terms, let $\Pi_\tau$ denote the temporal $L^2$-orthogonal projection onto piecewise constant functions with respect to the time partition. In other words, for each $v\in L^2(0,T)$, let $\Pi_\tau v\in L^2(0,T)$ is defined by $\Pi_\tau v|_{I_n}\coloneqq \frac{1}{\tau_n}\int_{I_n} v\mathrm{d}t$ for each time-interval $I_n$, $n\in\{1,\dots,N\}$. 
In a slight abuse of notation, we extend $\Pi_\tau$ without change of notation to functions in general Bochner spaces $L^2(0,T;W)$ where $W$ is any Banach space.
We now define the temporal data oscillation terms by
\begin{subequations}\label{eq:temporal_oscillations}
\begin{align}
[\osc_{\tau,1}]^2 &\coloneqq \int_0^T \norm{ (\mathrm{I}-\Pi_\tau)\left(F[\mht]-H[\nabla \uht]\right)}_{H^{-1}(\Omega)}^2\mathrm{d}t, 
\\
[\osc_{\tau,2}]^2& \coloneqq \int_0^T \norm{(\mathrm{I}-\Pi_\tau)G}_{H^{-1}(\Omega)}^2 + \norm{\mht (\mathrm{I}-\Pi_\tau)\dpH[\nabla \uht] }_{\Omega}^2 \mathrm{d}t,
\end{align}
\end{subequations}
where $\mathrm{I}$ denotes the identity operator.
To define the spatial data oscillation terms, we consider a fixed but arbitrary polynomial degree $\kappa \geq 0$ and we define the following spatial approximation operators. 
For each $K \in \T$ and $F \in \calF^I$, let $\Pi_{K,\kappa}:L^2(K) \to \mathcal{P}_\kappa(K)$ and $\Pi_{F,\kappa}:L^2(F) \to \mathcal{P}_\kappa(F)$ denote the $L^2$-projections onto polynomials of degree $\kappa \geq 0$ over $K$ and $F$ respectively.
For each element $K \in \T$ and time interval $n \in \{1,...,N\}$ we define the local data oscillation terms by
\begin{multline}
        [\osc_{K,n,\kappa,i}]^2 \coloneqq \sum_{K' \in \widetilde{\T}_K} \int_{I_n} \Big[ h_{K'}^2 \norm{r_{K',n,i}-\Pi_{K',\kappa} r_{K',n,i}}_{K'}^2 \\
        +  \sum_{F \in \calF_{K'}^I} h_F \norm{j_{F,n,i}-\Pi_{F,\kappa} j_{F,n,i}}_F^2 \Big]\dt \quad i \in \{1,2\}.
\end{multline}
For each $i  \in \{1,2\}$, we define the global spatial oscillation terms by~$[\osc_{\kappa,i}]^2\coloneqq \sum_{n=1}^N \sum_{K \in \T} [\osc_{K,n,\kappa,i}]^2$.

\subsection{\emph{A posteriori} error bounds}

The first main result, given in Theorem~\ref{theorem:reliability} below, shows the reliability of the estimator defined in~\eqref{eq:eta-total-def}, i.e.\ the estimator bounds from above the extended norm of the error, as defined in~\eqref{eq:error-norm-defs}, up to a multiplicative constant and temporal data oscillation terms. 

\begin{theorem}[Reliability]\label{theorem:reliability}
    Let $(\uht,\mht) \in \Vone \times \Vtwo$ be the numerical solution of the FEM scheme in~\eqref{eq:FEM-scheme-def} that satisfies \ref{H:stabilization_main}. Then it holds that
    \begin{equation}
      \norm{u-\uht}_{\Vone+Y} + \norm{m-\mht}_{\Vtwo + Y} \lesssim \eta(\uht,\mht) +  \sum_{i=1}^2  \osc_{\tau,i}, \label{eq:V+Y-reliability}
    \end{equation}
    where the hidden constant depends only on $d$, $\nu$, $L_H$, $L_{H_p}$, $L_F$, $L_S$, $M_\infty$, $\diam \Omega$, $T$, and the shape-regularity of $\T$.
\end{theorem}
The proof of Theorem~\ref{theorem:reliability} is deferred to Section~\ref{sec:reliability-proof} below.

The following two main results show the efficiency of the estimator, i.e.\ the estimator is bounded from above by the error, up to spatial data oscillation terms.
Theorem~\ref{theorem:local-efficiency} below shows the local efficiency of the temporal jump estimator and of the residual estimators.

\begin{theorem}[Local efficiency]\label{theorem:local-efficiency}
    Let $(\uht,\mht) \in \Vone \times \Vtwo$ be the numerical solution of the FEM scheme in~\eqref{eq:FEM-scheme-def}. For any $K \in \T$, $n \in \{ 1,...,N\}$, and polynomial degree $\kappa \geq 0$, the local temporal jump and residual estimators satisfy
    \begin{multline}\label{eq:local-efficiency}
        \sum_{i=1}^2 [\etaRecon{i,K,n}^2 + \etaRes{i,K,n}^2] \\
        \lesssim  \norm{\uht-\Uht}_{L^2(I_n;H^1(\omega_K))}^2 + \norm{\mht-\Mht}_{L^2(I_n;H^1(\omega_K))}^2 \\
        + \norm{\partial_t (u-\Uht)}_{L^2(I_n;H^{-1}(\omega_K))}^2 + \norm{\partial_t (m-\Mht)}_{L^2(I_n;H^{-1}(\omega_K))}^2 \\
     + \norm{u-\Uht}_{L^2(I_n;H^1( \omega_K))}^2  + \norm{m-\Mht}_{L^2(I_n;H^1(\omega_K))}^2  \\
     + \norm{F[m]-F[\mht]}_{L^2(I_n;H^{-1}(\omega_K))}^2+ \sum_{i=1}^2 [\osc_{K,n,\kappa,i}]^2, 
    \end{multline}
    where the hidden constant depends only on $d$, $\kappa$, $\nu$, $L_H$, $L_{H_p}$, $M_\infty$, $\diam \Omega$, and the shape-regularity of $\T$.
\end{theorem}
Theorem~\ref{theorem:local-efficiency} shows that the temporal jump and residual estimators are locally efficient over each time interval-element block $I_n \times K$.
The proof of Theorem~\ref{theorem:local-efficiency} is given in Section~\ref{sec:efficiency-proof}.
\begin{remark}[Nonlocality of coupling terms]
Note that since the coupling operator $F$ is allowed to be nonlocal, we do not attempt to bound further the terms $\norm{F[m]-F[\mht]}_{L^2(I_n;H^{-1}(\omega_K))}^2$ in~\eqref{eq:local-efficiency}. 
Nevertheless, in the global efficiency result below, we show that the sum of these terms over time intervals and mesh elements will be bounded from above by the error $m-\mht$, which is why we view this term as locally efficient.
\end{remark}
\begin{remark}[Initial and final time estimators]\label{rem:efficiency_initial_final}
The initial and final time estimators $\eta_0$ and $\eta_T$ defined in~\eqref{eq:eta0-T-defs} above are also locally efficient. Indeed, defining the local contributions $\eta_{0,K}\coloneqq \norm{m_0-\Pi_0 m_0}_K $ and $\eta_{T,K}\coloneqq \norm{S[\mht(T)]-\Pi_T S[\mht(T)]}_K$ for each $K\in\T$, we have the local efficiency bounds
\begin{subequations}\label{eq:initial_final_estimator_efficiency}
\begin{align}
\eta_{0,K} &=\norm{m(0)-\mht(0)}_K,
\\ \eta_{T,K} &\leq \norm{S[\mht(T)]-S[m(T)]}_K+\norm{u(T)-\uht(T)}_K,
\end{align}
\end{subequations}
for each $K\in\T$. Indeed, we obtain~\eqref{eq:initial_final_estimator_efficiency} from the triangle inequality and the identities $m(0)=m_0$, $u(T)=S[m(T)]$ and $\uht(T)=\Pi_T S[\mht(T)]$.
Since the operator $S$ is allowed to be nonlocal, we do not bound further the local terms $\norm{S[\mht(T)]-S[m(T)]}_K$.
\end{remark}

Note also that it is not possible to show local efficiency or computability of the stabilization estimator for the general abstract class of stabilizations considered above, since it is not locally computable.
However, we can show the global efficiency of all components of the estimator, as seen in Theorem~\ref{theorem:global-efficiency} below.

\begin{theorem}[Global efficiency] \label{theorem:global-efficiency}
    Let $(\uht,\mht) \in \Vone \times \Vtwo$ be the numerical solution of the FEM scheme in~\eqref{eq:FEM-scheme-def}. For any polynomial degree $\kappa \geq 0$, the \emph{a posteriori} error estimator satisfies
    \begin{equation*}
        \eta(\uht,\mht) \lesssim\norm{u-\uht}_{\Vone+Y} + \norm{m-\mht}_{\Vtwo + Y} +  \sum_{i=1}^2 \osc_{\kappa,i},
    \end{equation*}
    where the hidden constant depends only on $d$, $\kappa$, $\nu$, $L_H$, $L_{H_p}$, $L_F$, $L_S$, $M_\infty$, $\diam \Omega$, $T$, and the shape-regularity of $\T$.
\end{theorem}
The consequence of Theorem~\ref{theorem:global-efficiency} is that the total error estimator is bounded from above by the approximation error, proving that the estimator does not overestimate the error, up to a multiplicative constant.
Theorems~\ref{theorem:reliability} and~\ref{theorem:global-efficiency} together establish the equivalence between the error and the estimator, up to data oscillation.

\subsection{Proof of reliability}\label{sec:reliability-proof}
The starting point for the analysis is the equivalence result of Theorem~\ref{theorem:equivalence-result}, which implies that $ \norm{u-\uht}_X + \norm{m-\Mht}_Y \lesssim \calR(\uht,\Mht)$, where it is recalled that $\calR(\uht,\Mht)$ defined by~\eqref{eq:calR-def}.
In order to connect the residuals at the continuous level with the numerical scheme, it is helpful to write the definition of the FEM~\eqref{eq:FEM-scheme-def} in the equivalent form
\begin{equation}\label{eq:FEM-scheme-def-alt}
\begin{aligned}
\mathfrak{R}_i(\vht) = \mathcal{S}_i(\uht,\mht;\vht) \quad\forall \vht \in \V, \; \forall i\in \{1,2\},
\end{aligned}
\end{equation}
where the discrete residual functionals $\mathfrak{R}_i\in X^*$, $i\in \{1,2\}$, are defined by
\begin{align}
\mathfrak{R}_1(v) &\coloneqq 
\int_0^T \dualprod{F[\mht]}{v}{} \dt  \label{eq:discrete_R1_def}
\\
&- \int_0^T \big[-(\partial_t   \Uht, v)_\Omega + (\nu \nabla \uht,\nabla v)_\Omega + (H[\nabla \uht],v)_\Omega \big] \mathrm{d}t, \notag 
\\
 \mathfrak{R}_2(v)&\coloneqq \int_0^T \dualprod{G}{v}{} \dt \label{eq:discrete_R2_def}
 \\ 
 &-\int_0^T \big[ (\partial_t  \Mht,v)_{\Omega}  + \LTwoInner{\nu \nabla \mht}{\nabla v}{\Omega} + \LTwoInner{\mht \dpH[\nabla \uht]}{\nabla v}{\Omega} \big]\dt, \notag
\end{align}
for all $v\times X$.
By introducing the above discrete residual functionals, we can now give a suitable decomposition of the residuals $R_1^X(\uht,\Mht)$ and $R_1^Y(\uht,\Mht)$ that appear in $\calR(\uht,\Mht)$, as shown in the following Lemma.
\begin{lemma}[Decomposition of the residuals]\label{lem:residual_decomposition}
For each $v\in Y_0$, we have
\begin{multline}\label{eq:residual_decomposition_1}
\dualprod{R_1^X(\uht,\Mht)}{v}{Y_0^*\times Y_0} = \mathfrak{R}_1(v) 
 + \int_0^T \dualprod{\partial_t v}{\Uht - \uht}{} \mathrm{d}t 
 \\ + 
 \left((\mathrm{I}-\Pi_T)S[\mht(T)],v(T)\right)_\Omega + \int_0^T \dualprod{F[\Mht]-F[\mht]}{v}{}\mathrm{d}t,
\end{multline}
where $\mathrm{I}$ denotes the identity operator on $L^2(\Omega)$.
For each $w\in X$, we have
\begin{multline}\label{eq:residual_decomposition_2}
\dualprod{R_1^Y(\uht,\Mht)}{w}{X*\times X} = \mathfrak{R}_2(w) 
\\ + \int_0^T \left[ \nu(\nabla (\mht-\Mht),\nabla w)_\Omega + \left( (\mht-\Mht)\dpH[\nabla \uht],\nabla w\right)_\Omega  \right]\mathrm{d}t.
\end{multline}
\end{lemma}
\begin{proof}
To show~\eqref{eq:residual_decomposition_1}, we start from the definition of~$R_1^X(\uht,\Mht)$ in~\eqref{eq:R_1^X-def}, and we add and subtract the terms $\int_0^T \dualprod{\partial_t v}{\Uht}{}\mathrm{d}t$ and $\int_0^T \dualprod{F[\mht]}{v}{}\mathrm{d}t$.
We then obtain~\eqref{eq:residual_decomposition_1} by using the identity $\int_0^T \dualprod{\partial_t v}{\Uht}{}\mathrm{d}t = (\Pi_T S[\mht(T)],v(T))_\Omega-\int_0^T(\partial_t \Uht,v)_\Omega\mathrm{d}t $, which follows from the final time condition $\uht(T)=\Uht(T)=\Pi_T S[\mht(T)]$ and from $v(0)=0$ as $v\in Y_0$.
The proof of \eqref{eq:residual_decomposition_2} follows from a similar calculation, the details are left to the reader.
\end{proof}

The decompositions of the residuals of Lemma~\ref{lem:residual_decomposition} allows us to bound the dual norms $R_1^X(\uht,\Mht)$ and $R_2^Y(\uht,\Mht)$.
First, we show how the discrete residual functionals~$\mathfrak{R}_i$ that appear as the first terms in the decompositions lead to the residual and stabilization estimators, plus the temporal data oscillation.
\begin{lemma}\label{lem:residual_estimators}
For each $v\in X$, we have
\begin{equation}\label{eq:R-frac-i-dual-bound}
\abs{\mathfrak{R}_i(v)}\lesssim \left( \etaRes{i}+ \etaStab{i} + \osc_{\tau,i} \right)\norm{v}_X \quad i \in \{1,2\},
\end{equation}
where the hidden constants depend only on $d$ and the shape-regularity of $\T$.
\end{lemma}
\begin{proof}
Let $v\in X$ be arbitrary.
Recall that $\Pi_\tau$ denotes the global time-averaging projection, and let $\Pi_h\colon H^1_0(\Omega)\to V_h$ denote the Scott--Zhang quasi-interpolant operator~\cite{ScottZhang1990}.
Let $\Pi_{h,\tau}=\Pi_h \Pi_\tau :X \to \V$ be the composition of $\Pi_h$ and $\Pi_\tau$.
We start by decomposing 
\begin{equation}\label{eq:residual_individual_bounds_1}
    \mathfrak{R}_{i}(v) = \mathfrak{R}_{i}(\Pi_{h,\tau} v) +
    \mathfrak{R}_{i}(v-\Pi_\tau v)+ \mathfrak{R}_{i}(\Pi_\tau v-\Pi_{h,\tau} v) ,
\end{equation}
It is then seen from~\eqref{eq:FEM-scheme-def-alt}, which gives  $\mathfrak{R}_{i}(\Pi_{h,\tau} v)=\mathcal{S}_i(\uht,\mht;\Pi_{h,\tau} v)$, and from the $X$-norm stability of $\Pi_{h,\tau}$, that $\lvert\mathfrak{R}_{i}(\Pi_{h,\tau} v)\rvert \lesssim \etaStab{i}\norm{v}_X$.
The orthogonality of the projection~$\Pi_\tau$ implies that the second term on the right-hand side of~\eqref{eq:residual_individual_bounds_1} is bounded, up to a constant, by $\osc_{\tau,i} \norm{v}_X$, where we recall that the temporal oscillation terms~$\osc_{\tau,i}$ are defined in~\eqref{eq:temporal_oscillations}.
The bound for the final term on the right-hand side of~\eqref{eq:residual_individual_bounds_1} follows the usual approach of residual estimators, see~\cite[Theorem~3.58, p.~135]{verfurth2013posteriori}: we use integration by parts in space, together with the definitions of the volume residuals and face jumps in~\eqref{eq:volume-residual-defs} and~\eqref{eq:spatial-jump-residual-defs}, to find that 
\begin{equation}\label{eq:residual_IBP}
\mathfrak{R}_{i}(w) \\= \sum_{n=1}^N \int_{I_n} \left[  \sum_{K \in \T} \LTwoInner{r_{K,n,i}}{w}{K} -  \sum_{F \in \calF^I} \LTwoInner{j_{F,n,i}}{w}{F} \right] \dt \quad\forall w\in X.
\end{equation}
Then, using the identity~\eqref{eq:residual_IBP} for  $w=(\Pi_\tau-\Pi_{h,\tau})v = (\mathrm{I}-\Pi_h)\Pi_\tau v $, along with the approximation properties of $\Pi_h$, see~\cite{ScottZhang1990}, we deduce that 
$|\mathfrak{R}_{i}(\Pi_\tau v-\Pi_{h,\tau} v)|\lesssim \etaRes{i} \norm{v}_X$ for each $i\in\{1,2\}$.
This shows~\eqref{eq:R-frac-i-dual-bound}.
\end{proof}

\begin{lemma}[Bounds on residuals]\label{cor:residual_individual_bounds}
\begin{align}
    \norm{R_1^X(\uht,\Mht)}_{Y_0^*} &\lesssim \etaRes{1} +  \etaStab{1} + \etaRecon{1} + \eta_T + \etaRecon{2} + \osc_{\tau,1}, \label{eq:cor:residual_individual_bounds_1}
\\
    \norm{R_2^Y(\uht,\Mht)}_{X^*} &\lesssim  \etaRes{2} + \etaStab{2} + \etaRecon{2} + \osc_{\tau,2}. \label{eq:cor:residual_individual_bounds_2}
\end{align}
where the hidden constants depend only on $d$, $\nu$, $L_H$, $L_{H_p}$, $L_F$, $L_S$, $M_\infty$, $T$, $\diam \Omega$, and the shape-regularity of $\T$.
\end{lemma}
\begin{proof}
Lemma~\ref{lem:residual_estimators} provides a bound for the first term on the right-hand sides of~\eqref{eq:residual_decomposition_1} and~\eqref{eq:residual_decomposition_2}, so we turn our attention to the subsequent terms. 
Note in passing that $\norm{v}_X\leq \norm{v}_Y$ for all $v\in Y$.
It is clear that $\lvert \int_0^T \dualprod{\partial_t v }{\Uht-\uht}{}\mathrm{d}t\rvert \lesssim \norm{\Uht-\uht}_X \norm{v}_Y $ for all $v\in Y_0$.
Using the definition of $\norm{\cdot}_Y$, it is also clear that $\lvert\left((\mathrm{I}-\Pi_T)S[\mht(T)],v(T)\right)_\Omega\rvert \leq \eta_T \norm{v}_Y$.
Finally, the Lipschitz continuity of $F$ implies that $\lvert\int_0^T \dualprod{F[\Mht]-F[\mht]}{v}{}\mathrm{d}t\rvert \lesssim \etaRecon{2} \norm{v}_Y$.
This shows~\eqref{eq:cor:residual_individual_bounds_1}.
The proof of~\eqref{eq:cor:residual_individual_bounds_2} is shown in a similar manner, in particular by noting that the terms on the second line of~\eqref{eq:residual_decomposition_2} are bounded by $\etaRecon{2}\norm{w}_X$.
\end{proof}

We now complete the proof of Theorem~\ref{theorem:reliability}.
\begin{proof}[Proof of Theorem~\ref{theorem:reliability}]

By Hypothesis~\ref{H:stabilization_main} we have $(\uht,\Mht)\in X\times Y$ with $\Mht \geq 0$ in $Q_T$, so we apply Theorem~\ref{theorem:equivalence-result} and Lemma~\ref{cor:residual_individual_bounds} to deduce that
\begin{equation}\label{eq:XY-reliability}
        \norm{u-\uht}_X + \norm{m-\Mht}_Y \lesssim \mathcal{R}(\uht,\Mht) \lesssim  \eta(\uht,\mht)+\sum_{i=1}^2 \osc_{\tau,i} , 
\end{equation}
where the hidden constant depends only on $d$, $\nu$, $L_H$, $L_{H_p}$, $L_F$, $L_S$, $M_\infty$, $T$, $\diam \Omega$, and the shape-regularity of $\T$.
It is clear that $ \norm{m-\mht}_{\Vtwo + Y} = \norm{m-\Mht}_Y+\etaRecon{2}$ is also bounded by a constant times the right-hand side of \eqref{eq:XY-reliability}. 
Furthermore, the triangle inequality shows that $\norm{u-\Uht}_X \leq \norm{u-\uht}_X+\etaRecon{1}$, 
 and the identities $u(T)=S[m(T)]$ and $\Uht(T)=\Pi_T S[\Mht(T)]$ imply that $\norm{u(T)-\Uht(T)}_\Omega \lesssim \norm{m-\Mht}_Y+\eta_T$. 
 Therefore, we deduce that $\norm{u-\Uht}_X+\norm{(u-\Uht)(T)}_\Omega $ is also bounded by a constant times the right-hand side of~\eqref{eq:XY-reliability}.
 It remains only to bound $\norm{\partial_t(u-\Uht)}_{X^*}$.
 It follows from~\eqref{eq:HJB_time_strong_form} and the definition of $\mathfrak{R}_1$ in~\eqref{eq:discrete_R1_def} that
 \begin{multline}\label{eq:timederiv_u_Uht}
 \int_0^T \dualprod{\partial_t(\Uht - u )}{v}{}\mathrm{d}t = \mathfrak{R}_{1}(v)+\int_0^T \dualprod{F[m]-F[\mht]}{v}{} \mathrm{d}t 
 \\ + \int_0^T(\nu\nabla (\uht-u),\nabla v)_\Omega + (H[\nabla \uht]-H[\nabla u],v)_\Omega \mathrm{d}t \quad \forall v\in X. 
 \end{multline}
Therefore, the bound of Lemma~\ref{lem:residual_estimators} and~\eqref{eq:timederiv_u_Uht} above, along with the Lipschitz continuity of~$F$ and~$H$, imply that
\begin{equation}
    \norm{\partial_t(u-\Uht)}_{X^*}\lesssim \etaRes{1}+\etaStab{1}+\osc_{\tau,1}+\norm{m-\mht}_X + \norm{u-\uht}_X.
\end{equation}
We therefore deduce~\eqref{eq:V+Y-reliability} by combining the various bounds above.
\end{proof}

\subsection{Proof of efficiency}\label{sec:efficiency-proof}
We begin by proving local efficiency of the temporal jump and residual estimators.

\subsubsection{Proof of Theorem~\ref{theorem:local-efficiency}}
\begin{proof}
    Let $K \in \T$, $F \in \calF^I_K$, and $n \in \{1,...,N\}$. By the definition of the extended norm in~\eqref{eq:error-norm-defs}, we readily have an efficiency result for the local temporal jump estimator 
    \begin{equation} \label{eq:proof:eff:etaRec-bound}
    \sum_{i=1}^2 \etaRecon{i,K,n}^2\leq \norm{\uht-\Uht}_{L^2(I_n;H^1(\omega_K))}^2 + \norm{\mht-\Mht}_{L^2(I_n;H^1(\omega_K))}^2. 
\end{equation}
It remains to prove local efficiency of the residual estimator. Here we adapt the steady-state arguments of~\cite[Theorem 5.3]{OsborneSmearsWells2025} to the time-dependent setting. Since $r_{K,n,i} \in L^2(K)$ and $j_{F,n,i} \in L^2(F)$ for $t$ a.e.\ in $I_n$, it holds from standard spatial bubble function arguments~\cite[Lemma 5.1]{OsborneSmearsWells2025} that
\begin{equation}\label{eq:proof:eff:bubble-func-result}
\begin{aligned}
    h_K^2 \norm{r_{K,n,i}}_{L^2(I_n \times K)}^2 &\lesssim \int_{I_n} \norm{r_{K,n,i}}_{H^{-1}(K)}^2 + h_K^2 \norm{r_{K,n,i}-\Pi_{K,\kappa} r_{K,n,i}}_K^2\dt, \\
    h_F \norm{j_{F,n,i}}_{L^2(I_n \times F)}^2 &\lesssim \int_{I_n} \left[ \sup_{w \in H^1_0(\omega_F) \backslash \{0\}} \frac{\LTwoInner{j_{F,n,i}}{w}{F}}{\norm{\nabla w}_{\omega_F}} \right]^2  + h_F \norm{j_{F,n,i}-\Pi_{F,\kappa} j_{F,n,i}}_F^2\dt.
\end{aligned}
\end{equation}
Subtracting a local weak formulation of the MFG system (in which the time derivative is not cast onto the test function), bounding the resulting terms with the triangle, Cauchy--Schwarz, and the Poincar\'e inequalities, then applying Lipschitz continuity arguments of $H$, $\dpH$, $F$, we obtain
\begin{multline}
    \sum_{i=1}^2 \int_{I_n} \norm{r_{K,n,i}}_{H^{-1}(K)}^2\dt
    \\
    \lesssim \norm{u-\uht}_{L^2(I_n;H^1(K))}^2  + \norm{m-\mht}_{L^2(I_n;H^1(K))}^2 \\
    +\norm{\partial_t (u-\Uht)}_{L^2(I_n;H^{-1}(K))}^2 + \norm{\partial_t (m-\Mht)}_{L^2(I_n;H^{-1}(K))}^2 \\
      + \norm{F[m]-F[\mht]}_{L^2(I_n;H^{-1}(K))}^2, \label{eq:proof:eff:rK-bound}
\end{multline}
with a hidden constant dependent only on $d$, $\kappa$, $\nu$, $L_H$, $L_{H_p}$, $M_\infty$, and the shape-regularity of $\T$.
Note that in deriving the above bound we have applied the bound $h_K \lesssim \diam \Omega$ to absorb higher powers of $h_K$ into the hidden constant. Following a similar argument for the jump terms, applying integration by parts over $\omega_F$ and subtracting a local weak formulation of the MFG system and bounding error terms we have 
\begin{multline}
    \sum_{i=1}^2 \int_{I_n} \left[ \sup_{w \in H^1_0(\omega_F) \backslash \{0\}} \frac{\LTwoInner{j_{F,n,i}}{w}{F}}{\norm{\nabla w}_{\omega_F}} \right]^2\dt \lesssim \sum_{i=1}^2 \int_{I_n} \norm{r_{K,n,i}}_{H^{-1}(\omega_F)}^2\dt \\
    + \norm{\partial_t (u-\Uht)}_{L^2(I_n;H^{-1}(\omega_F))}^2 + \norm{\partial_t (m-\Mht)}_{L^2(I_n;H^{-1}(\omega_F))}^2 \\
     + \norm{u-\uht}_{L^2(I_n;H^1(\omega_F))}^2  + \norm{m-\mht}_{L^2(I_n;H^1(\omega_F))}^2 \\
      + \norm{F[m]-F[\mht]}_{L^2(I_n;H^{-1}(\omega_F))}^2. \label{eq:proof:eff:jF-bound}
\end{multline}
Next, we bound the local residual estimator by applying the bounds of~\eqref{eq:proof:eff:bubble-func-result},~\eqref{eq:proof:eff:rK-bound} and~\eqref{eq:proof:eff:jF-bound} over sums of elements $K \subset \omega_K$. This yields 
\begin{multline}
    \sum_{i=1}^2 \etaRes{K,n,i}^2 \lesssim  \norm{u-\uht}_{L^2(I_n;H^1( \omega_K))}^2  + \norm{m-\mht}_{L^2(I_n;H^1(\omega_K))}^2 \\
    +\norm{\partial_t (u-\Uht)}_{L^2(I_n;H^{-1}(\omega_K))}^2 + \norm{\partial_t (m-\Mht)}_{L^2(I_n;H^{-1}(\omega_K))}^2 \\
     + \norm{F[m]-F[\mht]}_{L^2(I_n;H^{-1}(\omega_K))}^2 + \sum_{i=1}^2 [\osc_{K,n,\kappa,i}]^2.\label{eq:proof:eff:etaRes-bound}
\end{multline}
Combining~\eqref{eq:proof:eff:etaRec-bound} and~\eqref{eq:proof:eff:etaRes-bound} concludes the proof.
\end{proof}

\subsubsection{Proof of Theorem~\ref{theorem:global-efficiency}}
\begin{proof}
    By summing the bounds in~\eqref{eq:local-efficiency} over each $K \in \T$ and $n \in \{1,...,N\}$, and using the Poincar\'e inequality, we find that
    \begin{multline*}
        \sum_{i=1}^2 \left[\etaRecon{i}^2 + \etaRes{i}^2 \right] \lesssim \norm{\uht-\Uht}_{X}^2 + \norm{\mht-\Mht}_X^2 \\
        + \norm{u-\Uht}_Y^2 + \norm{m-\Mht}_Y^2 + \norm{F[m]-F[\mht]}_{X^*}^2 +\sum_{i=1}^2 [\osc_{\kappa,i}]^2,
    \end{multline*}
    where we have also used the well-known inequality $\sum_{K \in \T} \norm{\Phi}_{H^{-1}(\omega_K)}^2 \lesssim \norm{\Phi}_{H^{-1}(\Omega)}^2$ for all $\Phi\in H^{-1}(\Omega)$,  where the hidden constant depends only on~$d$ and the shape-regularity of~$\T$.
   Next, we apply Young's inequality, the definition of the extended norms in~\eqref{eq:error-norm-defs}, the Lipschitz continuity of $F$, and the embedding $Y\hookrightarrow Z$, to obtain
    \begin{equation}
        \sum_{i=1}^2 \left[\etaRecon{i} + \etaRes{i} \right] \lesssim\norm{u-\uht}_{\Vone+Y} + \norm{m-\mht}_{\Vtwo + Y}  +\sum_{i=1}^2\osc_{\kappa,i}, \label{eq:proof:globaleff:etaRes-bound}
    \end{equation}
    with a hidden constant dependent only on $d$, $\kappa$, $\nu$, $L_H$, $L_{H_p}$, $L_F$, $M_\infty$, $\diam \Omega$, and the shape-regularity of $\T$.
    The initial and final time estimators are bounded by summing the bounds from~\eqref{eq:initial_final_estimator_efficiency} over all mesh elements, and applying the embedding~\eqref{eq:Y-embedding} as well as the Lipschitz continuity of $S$, which leads to the bound
    \begin{equation}
        \eta_0 + \eta_T \lesssim\norm{u-\uht}_{\Vone+Y} + \norm{m-\mht}_{\Vtwo + Y}, \label{eq:proof:globaleff:eta0T-bound}
    \end{equation}
    where the hidden constant depends only on $L_S$, $\diam \Omega$, and $T$.
   It remains only to bound the stabilization estimators defined in~\eqref{eq:etaStab-defs}. By subtracting the strong forms of the HJB and KFP equations from the FEM scheme in~\eqref{eq:FEM-scheme-def} and applying integration by parts in space, we have
    \begin{align*}
        \mathcal{S}_1(\uht;\mht;\vht) &= \int_0^T \LTwoInner{F[\mht]-F[m] + H[\nabla u]-H[\nabla \uht]}{\vht}{\Omega}\dt \nonumber \\
        &\hspace{-1cm} +\int_0^T \dualprod{\partial_t(\Uht-u)}{\vht}{} + \LTwoInner{\nu \nabla (u-\uht)}{\nabla \vht}{\Omega}\dt, \\
        \mathcal{S}_2(\uht;\mht;\wht) &= \int_0^T \dualprod{\partial_t (m-\Mht)}{\vht}{}\dt \nonumber \\
        &\hspace{-1cm} +\int_0^T \LTwoInner{\nu \nabla (m-\mht) + (m-\mht) \dpH[\nabla \uht]}{\nabla \wht}{\Omega}~dt.
    \end{align*}
    Then by repeating continuity arguments (cf. the proof of Lemma~\ref{lemma:residual-bounds}), we obtain
    \begin{equation}\label{eq:etaStab-bound-by-error}
    \begin{aligned}
        \etaStab{1} &\lesssim \norm{u-\Uht}_Y + \norm{u-\uht}_X + \norm{m-\mht}_X, \\
        \etaStab{2} &\lesssim \norm{m-\Mht}_Y + \norm{m-\mht}_X,
    \end{aligned}
    \end{equation}
    where the hidden constants depend only on $\nu$, $L_H$, $L_F$, and $\diam \Omega$. Then we apply the triangle inequality and the definition of the extended norms to get
    \begin{equation}
        \sum_{i=1}^2 \etaStab{i} \lesssim\norm{u-\uht}_{\Vone+Y} + \norm{m-\mht}_{\Vtwo + Y}. \label{eq:proof:globaleff:etaStab-bound}
    \end{equation}
    The proof is concluded by summing the global efficiency estimates~\eqref{eq:proof:globaleff:etaRes-bound}, \eqref{eq:proof:globaleff:eta0T-bound}, and~\eqref{eq:proof:globaleff:etaStab-bound}.
\end{proof}

\section{Stabilization-free error bounds}\label{sec:local-estimators}
We conclude the analysis by showing that under additional assumptions on the FEM scheme in~\eqref{eq:FEM-scheme-def}, the \emph{a posteriori} error bound does not require explicit computation of the stabilization estimators. 
To demonstrate that this is attainable, we consider a general class of stabilization schemes defined by
\begin{equation} \label{eq:Stabfree-defs}
    \begin{aligned}
    \calS_1(\uht;\mht;\vht) &\coloneqq \int_0^T \LTwoInner{\partial_t \Uht}{\vht}{\Omega} - \LTwoInner{\partial_t \Uht}{\vht}{\Omega,h} + \mathcal{S}_{h,1}(\uht; \mht; \vht)\dt, \\
    \calS_2(\uht;\mht;\wht) &\coloneqq \int_0^T \LTwoInner{\partial_t \Mht}{\wht}{\Omega,h} - \LTwoInner{\partial_t \Mht}{\wht}{\Omega}+\mathcal{S}_{h,2}(\uht; \mht; \vht)\dt,
\end{aligned}
\end{equation}
where $\mathcal{S}_{h,i}:[V_h]^3  \to \R$ denotes a spatial stabilization that is allowed to be nonlinear in its first two arguments, but is linear with respect to its third argument.

The main result of this section, presented in Theorem~\ref{theorem:stab-free-error-bound} below, is a stabilization-free \emph{a posteriori} error bound which is independent of the stabilization estimator and efficient, up to data oscillation terms. The analysis requires conditions on the discretization scheme and stabilization terms. Including the nonnegativity Hypothesis~\ref{H:stabilization_main}, we require four additional conditions: the time-discretization is larger than the squared mesh size (up to a hidden constant)~\ref{H:ht-ratio}, the spatial stabilization is affine-preserving~\ref{H:affine-preserving} and Lipschitz continuous~\ref{H:Lipschitz}, and the density approximation is uniformly bounded~\ref{H:Mht-infty}.

\subsection{Mass lumping error bound}
In this section, we show that the component of the stabilization estimator that results from lumping the mass matrix is bounded by the temporal jump estimators.
A key ingredient of the analysis is an original explicit formula for the mass lumping error given in Lemma~\ref{lemma:mass-lumping-expansion} below, which is of independent interest.
Recall that, for each edge $E \in \calE$ of the mesh, the vector~$t_E$ denotes a chosen tangent vector to the edge; the orientation of $t_E$ is of no consequence in the following.
\begin{lemma}[Formula for mass lumping]\label{lemma:mass-lumping-expansion}
   The mass lumping inner product $\LTwoInner{\cdot}{\cdot}{\Omega,h}$ defined in~\eqref{eq:mass-lump-def} satisfies
   \begin{equation}\label{eq:mass-lumping-diffusion-formula}
    \LTwoInner{v_h}{w_h}{\Omega,h} = \LTwoInner{v_h}{w_h}{\Omega} + \LTwoInner{D_h^M \nabla v_h}{\nabla w_h}{\Omega} \quad \forall v_h,w_h \in V_h,
\end{equation}
where the piecewise constant matrix-valued function $D_h^M \in  L^{\infty}(\Omega;\R^{d\times d}_{\mathrm{sym}})$ is defined elementwise by
    \begin{equation} \label{eq:D_h^M-def}
        D_h^M|_K \coloneqq \frac{1}{(d+1)(d+2)} \sum_{E \in \calE_K} h_E^2~ t_E \otimes t_E \quad \forall K \in \T.
    \end{equation}
\end{lemma}
\begin{proof}
By linearity, it is sufficient to show that \eqref{eq:mass-lumping-diffusion-formula} holds for all pairs of nodal basis functions $\{\varphi_z\}_{z\in \calV^I}$ of $V_h$.
    For each element $K \in \T$, let the local contribution to the mass-lumped inner product be defined by $\LTwoInner{v_h}{w_h}{K,h} \coloneqq \sum_{z \in \calV_K} \LTwoInner{\varphi_z}{1}{K} v_h(z)w_h(z)$, for all $v_h$, $w_h\in V_h$.
    Let $z_1,z_2 \in \calV_K$ be two vertices of $K$. 
    It is known from~\cite[Eq.~(30.3), p.~86]{ErnGuermond2021ii} that
    \begin{equation} \label{eq:proof:mass-lump-diff-1}
        \LTwoInner{\varphi_{z_1}}{\varphi_{z_2}}{K,h} = \frac{\abs{K}}{(d+1)} \delta_{z_1 z_2}, \quad \LTwoInner{\varphi_{z_1}}{\varphi_{z_2}}{K} = \frac{(1+\delta_{z_1 z_2})\abs{K}}{(d+1)(d+2)},
    \end{equation}
    where $\delta_{z_1z_2}$ denote the Kronecker delta.
 Then, recalling that functions in $V_h$ are piecewise affine, we have the identity $h_E \nabla \varphi_{z_i}|_K\cdot t_E = \varphi_{z_i}(z_E)-\varphi_{z_i}(z_E^\prime)$ for each edge $E\in \calV_K$ and $i\in\{1,2\}$, where $z_{E}$ and $z_{E}^\prime$ denote the endpoints of the edge $E$ such that $z_E-z_E^\prime=h_E t_E$.
Therefore, we deduce that
    \begin{equation}\label{eq:proof:mass-lump-diff-2}
        \begin{split}
    \LTwoInner{D_h^M \nabla \varphi_{z_1}}{\nabla \varphi_{z_2}}{K} & = \frac{\abs{K}}{(d+1)(d+2)} \sum_{E \in \calE_K}  (h_E\nabla \varphi_{z_1} \cdot t_E) (h_E\nabla \varphi_{z_2} \cdot t_E)
    \\ & = \frac{\abs{K}}{(d+1)(d+2)}\left((d+1)\delta_{z_1z_2} -1\right),
    \end{split}
    \end{equation}
    where we obtain the second line above by noting that, in the case $z_1=z_2$, there are precisely $d$ edges in $\calE_K$ that contain the vertex $z_1$, whereas in the case $z_1\neq z_2$, the product $(h_E\nabla \varphi_{z_1} \cdot t_E) (h_E\nabla \varphi_{z_2} \cdot t_E)$ is nonvanishing if and only if $E$ is the edge formed by the vertices $z_1$ and $z_2$.
    By comparing~\eqref{eq:proof:mass-lump-diff-1} with \eqref{eq:proof:mass-lump-diff-2}, we deduce that $ \LTwoInner{\varphi_{z_1}}{\varphi_{z_2}}{K,h} = \LTwoInner{\varphi_{z_1}}{\varphi_{z_2}}{K} + \LTwoInner{D_h^M \nabla \varphi_{z_1}}{\nabla \varphi_{z_2}}{K}$ for all $z_1$, $z_2\in \calV_K$. By summing the previous identity over all mesh elements, we deduce that~\eqref{eq:mass-lumping-diffusion-formula} holds for all pairs of basis functions of $V_h$, and thus, by linearity, for all pairs of functions in $V_h$.
\end{proof}
\begin{remark}\label{rem:masslumping_extension}
Lemma~\ref{lemma:mass-lumping-expansion} is of independent interest since it shows that the mass lumped inner-product defined on $V_h$ can be extended in a consistent and bounded way to all functions in $H^1(\Omega)$, including functions that are not sufficiently regular to apply the usual nodal definition from~\eqref{eq:mass-lump-def}.
\end{remark}

We now prove that the mass lumping error, when applied to the time reconstruction of a function, is bounded up to constants by the temporal jump estimator,
under the condition that each time-step size is bounded from below by a constant times the square of the maximum mesh size:
\begin{enumerate}[label={(H\arabic*)},resume]
    \item \label{H:ht-ratio} There exists a constant $C_{\T \mathcal{J}} > 0$ such that the discretization parameters satisfy $\tau_n \geq C_{\T \mathcal{J}} h_K^2$ for all $n \in \{1,...,N\}$ and $K \in \T$.
\end{enumerate}
Hypothesis~\ref{H:ht-ratio} is usually satisfied in practical computations, and it permits locally refined spatial meshes and long time-steps.
This hypothesis is not to be confused with the much more restrictive assumption of a parabolic CFL condition.

\begin{proposition}[Bound on mass lumping terms]\label{prop:mass-lump-estimate}
    Assume~\ref{H:ht-ratio}. Then, for every $\vht \in \Vi$, we have
    \begin{equation}\label{eq:mass-lump-estimate}
        \sup_{\wht \in \V \setminus \{0\}} \frac{\int_0^T \LTwoInner{\partial_t \calI_i \vht}{\wht}{\Omega} - \LTwoInner{\partial_t \calI_i \vht}{\wht}{\Omega,h}\dt}{\norm{\wht}_X} \lesssim \norm{\vht-\calI_{i} \vht}_X,
    \end{equation}
    where the hidden constant depends only on $C_{\T \mathcal{J}}$ defined in Hypothesis~\ref{H:ht-ratio} above. 
\end{proposition}
\begin{proof}
 We start by applying Lemma~\ref{lemma:mass-lumping-expansion} and the Cauchy--Schwarz inequality to obtain
    \begin{multline}\label{eq:proof:mass-diff-X-norm}
            \sup_{\wht \in \V \backslash \{0\}}  \frac{\int_0^T \LTwoInner{\partial_t \calI_i \vht}{\wht}{\Omega} - \LTwoInner{\partial_t \calI_i \vht}{\wht}{\Omega,h}\dt}{\norm{\wht}_X} \\
            \leq \norm{D_h^M}_{L^\infty(Q_T;\R^{d \times d})} \norm{\partial_t \calI_i \vht}_X
            \lesssim \max_{K\in\T}h_K^2 \norm{\partial_t \calI_i \vht}_X.
    \end{multline}
Next, we relate $\partial_t \calI_i \vht$ with the difference $\calI_i \vht -\vht$. 
By definition of $\calI_i \vht$ in~\eqref{eq:Reconstruction-defs}, we have, for each $n \in \{1,...,N\}$ and time interval $I_n$, the identities $(t-t_{n-1})\partial_t \calI_1 \vht = \calI_1 \vht -\vht$ if $\vht\in \Vone$ and $(t-t_n)\partial_t \calI_2 \vht = \calI_2 \vht -\vht$ if $\vht\in \Vtwo$.
Therefore, after integrating  over $I_n\times \Omega$ and using the fact that $\partial_t \calI_i \vht$ is constant in time on $I_n$, we deduce from the above identities that
    \begin{equation}\label{eq:proof:time-grad-identity}
        \int_{I_n} \norm{\nabla \partial_t \calI_i \vht}_\Omega^2\dt = \frac{3}{\tau_n^2} \int_{I_n} \norm{\nabla (\calI_i \vht - \vht)}^2_\Omega\dt \quad \forall n\in\{1,\ldots,N\}.
    \end{equation}
It is then clear that~\eqref{eq:mass-lump-estimate} follows from \eqref{eq:proof:mass-diff-X-norm} and \eqref{eq:proof:time-grad-identity} under Hypothesis~\ref{H:ht-ratio}.
\end{proof}

\subsection{Spatial stabilization}
The analysis of the remaining spatial stabilization makes use of our previous results on affine-preserving spatial stabilization for steady-state MFG \cite{OsborneSmearsWells2025}.
 Here we denote the patch of elements containing a given vertex $z \in \calV^I$ by $\omega_z$.
This analysis requires three additional conditions, which are patchwise affine preservation, Lipschitz continuity, and uniform boundedness of the discrete densities.
\begin{enumerate}[label={(H\arabic*)},resume]
\item \label{H:affine-preserving} The stabilizations are patchwise affine-preserving: for every interior vertex $z \in \calV^I$ and each $i \in \{1,2\}$, it holds that $\calS_{h,i}(v_h; \widetilde{v}_h;\varphi_z) = 0$ whenever $v_h$ and $\widetilde{v}_h$ are affine functions over $\omega_z$, i.e.\ $v_h|_{\omega_z} \in \mathcal{P}_1(\omega_z)$ and $\widetilde{v}_h|_{\omega_z} \in \mathcal{P}_1(\omega_z)$.
\item \label{H:Lipschitz} The spatial stabilization scheme $\calS_{h,i}(\cdot,\cdot)$ satisfies for each $i \in \{1,2\}$
\begin{equation*}
\begin{aligned}
    \abs{\calS_{h,i}(v_h;w_h;\varphi_z) - \calS_{h,i}(\widetilde{v}_h; \widetilde{w}_h;\varphi_z)}& \\
    \leq L_{\calS} \abs{\omega_z}_d^{\frac{1}{2}} &\left( \norm{\nabla (v_h - \widetilde{v}_h)}_{\omega_z} + \norm{\nabla (w_h - \widetilde{w}_h)}_{\omega_z} \right), 
\end{aligned}
\end{equation*}
for all $v_h, w_h, \widetilde{v}_h, \widetilde{w}_h \in V_h$ and $z \in \calV^I$, for some bounded constant $L_{\calS} > 0$ that is independent of the discretization parameters. 
\item \label{H:Mht-infty} The density approximation $\mht \in \Vtwo$ satisfies 
\begin{equation*}
    \norm{\mht}_{L^\infty(Q_T)} \leq \widetilde{M}_\infty,
\end{equation*}
for some bounded constant $\widetilde{M}_\infty > 0$ that is independent of the discretization parameters. 
\end{enumerate}
We refer the reader to~\cite[Section~6]{OsborneSmearsWells2025} for some examples of stabilizations that verify the hypotheses above.
Under these conditions, we can bound the spatial component of the stabilization scheme by the jump components of the residual estimators.

\begin{proposition}[Estimation of spatial stabilization]\label{prop:stab<=jump}
Let $(\uht,\mht) \in \Vone \times \Vtwo$ be the numerical solution to the FEM scheme~\eqref{eq:FEM-scheme-def} computed using the stabilization scheme~\eqref{eq:Stabfree-defs}. If Hypotheses~\ref{H:affine-preserving}, \ref{H:Lipschitz}, and \ref{H:Mht-infty} hold, then it follows that
\begin{equation}
    \sum_{i=1}^2 \left[ \sup_{v_{{h,\tau}} \in\V\setminus\{0\}} \frac{\int_0^T \calS_{h,i}(\uht; \mht; v_{{h,\tau}}) \dt}{\norm{v_{{h,\tau}}}_X} \right] \lesssim \sum_{i=1}^2 \etaRes{i}, \label{eq:stab<=jump}
\end{equation}    
where the hidden constant depends on $d$, $\nu$, $L_{H_p}$, $L_{\calS}$, $\widetilde{M}_\infty$, and the shape-regularity of $\T$.
\end{proposition}
\begin{proof}
    First we apply the results from the steady-state case of \cite[Lemma 6.1 \& Theorem 6.4]{OsborneSmearsWells2025}, which use Hypotheses~\ref{H:affine-preserving} and~\ref{H:Lipschitz}, and \ref{H:Mht-infty}, to obtain the bound 
    \begin{equation}
         \calS_{h,i}(\uht|_{I_n};\mht|_{I_n};v_h) \lesssim \left( \sum_{i=1}^2 h_F \norm{j_{F,n,i}}_{F}^2  \right)^{\frac{1}{2}} \norm{\nabla  v_h}_\Omega, \quad \forall v_h \in V_h,  \label{eq:proof:stab<=jump-1}
    \end{equation}
     for each $n\in \{1,\ldots, N\}$ and each $i\in\{1,2\}$, with a hidden constant that depends on $d$, $\nu$, $L_{H_p}$, $L_{\calS}$, $\widetilde{M}_\infty$, and the shape-regularity of $\T$.
    The bound~\eqref{eq:stab<=jump} is then obtained by summing~\eqref{eq:proof:stab<=jump-1} over all time intervals, applying the Cauchy--Schwarz inequality, and recalling the definition of the residual estimator from~\eqref{eq:eta-res-defs}.
\end{proof}

\subsection{A posteriori bounds}
Our last main result in Theorem~\ref{theorem:stab-free-error-bound} below shows that, under the above hypotheses, we can obtain a locally computable estimator in which the stabilization estimators can be removed.

\begin{theorem}[stabilization-free \emph{a posteriori} error bound]\label{theorem:stab-free-error-bound}
    Let $(\uht,\mht) \in \Vone \times \Vtwo$ be the numerical solution of the FEM scheme~\eqref{eq:FEM-scheme-def} computed using a stabilization scheme of form~\eqref{eq:Stabfree-defs}. If Hypotheses~\ref{H:stabilization_main}, \ref{H:ht-ratio}, \ref{H:affine-preserving}, \ref{H:Lipschitz}, and \ref{H:Mht-infty} hold, then 
    \begin{equation}\label{eq:stab-free-error-bound}
    \begin{aligned}
          \norm{u-\uht}_{\Vone +Y} + \norm{m-\mht}_{\Vtwo + Y} &\lesssim \sum_{i=1}^2 [\etaRecon{i} + \etaRes{i}] +\eta_0 + \eta_T+ \sum_{i=1}^2 \osc_{\tau,i},
    \end{aligned}
       \end{equation}
     where the hidden constant depends on $d$, $\nu$, $L_H$, $L_{H_p}$, $L_F$, $L_S$, $M_\infty$, $C_{\T \mathcal{J}}$, $L_{\calS}$, $\widetilde{M}_\infty$, $\diam \Omega$, $T$, and the shape-regularity of $\T$.
\end{theorem}
\begin{proof}
As a consequence of Propositions~\ref{prop:mass-lump-estimate} and~\ref{prop:stab<=jump}, the stabilization estimators are bounded in terms of the residual and temporal jump estimators, i.e.\
\begin{equation}
   \sum_{i=1}^2 \etaStab{i} \lesssim \sum_{i=1}^2 \left[\etaRes{i} + \etaRecon{i} \right]. \label{eq:etaStab<=etaRes+etaRec}
\end{equation}
The bound~\eqref{eq:stab-free-error-bound} is then obtained directly from~\eqref{eq:etaStab<=etaRes+etaRec} and Theorem~\ref{theorem:reliability}. 
\end{proof}
Note that the local efficiency of the residual and temporal jump estimators was already shown above in Theorem~\ref{theorem:local-efficiency}, and the local efficiency of the initial and final time estimators was shown in Remark~\ref{rem:efficiency_initial_final}.
Therefore, the estimator on the right-hand side of~\eqref{eq:stab-free-error-bound} is both locally computable and locally efficient.


\section*{Conclusion}
We obtained the first \emph{a posteriori} error bounds for a general class of stabilized finite element approximations of time-dependent mean 
 field games. 
The estimators are reliable and efficient for a very general class of stabilizations.
Under some additional structural assumptions on the stabilization and a weak condition on the discretization parameters, we also showed that the stabilization estimators are bounded by the temporal and spatial jump estimators, resulting in locally computable and locally efficient estimators.

\bibliographystyle{siamplain_NoURL} 
\bibliography{references}

\end{document}